\documentclass[12pt,reqno]{amsart}
\usepackage{amssymb}
\usepackage{amsfonts}
\usepackage{amssymb,latexsym}
\usepackage{enumerate}
\usepackage{mathrsfs}
\usepackage{url}
\usepackage{tikz}
\usetikzlibrary{calc, arrows.meta}
\allowdisplaybreaks

\makeatletter
\@namedef{subjclassname@2020}{%
  \textup{2020} Mathematics Subject Classification}
\makeatother

  [2002/01/19 v2.2g %
    AMS font definitions%
  ]
\DeclareFontFamily{U}{euf}{}
\DeclareFontShape{U}{euf}{m}{n}{%
  <5><6><7><8><9>gen*eufm%
  <10><10.95><12><14.4><17.28><20.74><24.88>eufm10%
  }{}
\DeclareFontShape{U}{euf}{b}{n}{%
  <5><6><7><8><9>gen*eufb%
  <10><10.95><12><14.4><17.28><20.74><24.88>eufb10%
  }{}

  [2002/01/19 v2.2g %
    AMS font definitions%
  ]
\DeclareFontFamily{U}{msb}{}
\DeclareFontShape{U}{msb}{m}{n}{%
  <5><6><7><8><9>gen*msbm%
  <10><10.95><12><14.4><17.28><20.74><24.88>msbm10%
  }{}

  [2002/01/19 v2.2g %
    AMS font definitions%
  ]
\DeclareFontFamily{U}{msa}{}
\DeclareFontShape{U}{msa}{m}{n}{%
  <5><6><7><8><9>gen*msam%
  <10><10.95><12><14.4><17.28><20.74><24.88>msam10%
  }{}

\newtheorem{theorem}{Theorem}[section]
\newtheorem{lemma}[theorem]{Lemma}
\newtheorem{proposition}[theorem]{Proposition}

\newtheorem{corollary}[theorem]{Corollary}

\theoremstyle{definition}
\newtheorem{remark}[theorem]{Remark}
\newtheorem{definition}[theorem]{Definition}

\def\t{\tilde}

\numberwithin{equation}{section} 

\begin{document}

\title[Ramanujan-type identities]
{Ramanujan-type identities for alternating Hurwitz zeta functions}
\author{Meng Yuan}
	\address{Department of Mathematics, South China University of Technology, Guangzhou, Guangdong 510640, China}
	\email{2714115098@qq.com}
	\author{Su Hu}
	\address{Department of Mathematics, South China University of Technology, Guangzhou, Guangdong 510640, China}
	\email{mahusu@scut.edu.cn}
	\author{Min-Soo Kim}
	\address{Department of Mathematics Education, Kyungnam University, Changwon, Gyeongnam 51767, Republic of Korea}
	\email{mskim@kyungnam.ac.kr}



%\thanks{This work is supported by Kyungnam University Foundation Grant, 2008.}

%    General info


\begin{abstract}
Around 1910, in an unpublished manuscript, Ramanujan proposed the following identity for $\zeta(2n+1)$:
\[
\begin{aligned}
\alpha^{-n}\,&\left\{\dfrac{1}{2}\,\zeta(2n + 1) + \sum_{m = 1}^{\infty}\dfrac{m^{-2n - 1}}{e^{2\alpha m} - 1}\right\} \\
&\quad\quad\quad\quad\quad\quad\quad-(-\beta)^{-n}\,\left\{\dfrac{1}{2}\,\zeta(2n + 1) + \sum_{m = 1}^{\infty}\dfrac{m^{-2n - 1}}{e^{2\beta m} - 1}\right\}\\
&=2^{2n}\sum_{k = 0}^{n + 1}\dfrac{(-1)^{k-1}B_{2k}\,B_{2n - 2k + 2}}{(2k)!(2n - 2k + 2)!}\,\alpha^{n - k + 1}\beta^k,
\end{aligned}
\]
where $\alpha$, $\beta$ are positive numbers satisfying $\alpha\beta=\pi^2,n\in\mathbb{N},$ $B_n$ denotes the $n$-th Bernoulli number and $\zeta(z)$ is the Riemann zeta function. As shown by Berndt in the viewpoint of general transformation
 of analytic Eisenstein series, it is a natural companion of  Euler's famous formula for even zeta values.
 
In this paper, we extend Ramanujan's identity to the alternating Hurwitz zeta function. Then we systematically investigate the properties of the alternating Hurwitz zeta function $\zeta_E(z,x)$, as well as the corresponding Ramanujan-type identities, under different modular symmetry conditions. We also establish infinite series expressions for products of the tangent and hyperbolic tangent functions, and express the Dirichlet lambda function $\lambda(z)$ together with linear combinations of infinite series as convolution sums of special sequences. Furthermore, we define alternating Hurwitz kernels of even and odd orders, and obtain Ramanujan-type identities involving the alternating digamma function $\widetilde{\psi}(x)$ and Euler polynomials $E_n(x)$, as well as transformation formulas between even-order and odd-order alternating Hurwitz kernels.
\end{abstract}
 \maketitle

\section{Introduction}
\subsection{History of the subject}
The Riemann zeta function $\zeta(z)$ is defined as
\begin{align}\label{Rzeta}
    \zeta(z) = \sum_{n=1}^{\infty} \frac{1}{n^z},
\end{align}
where $\operatorname{Re}(z) > 1$ (see \cite{cohen2007number, titchmarsh1986theory}).
It is one of the central objects of study in analytic number theory, connecting the distribution of prime numbers, complex analysis, and algebraic geometry. It admits the following Euler product expansion
\begin{equation}\label{Euler}
    \zeta(z) = \prod_{p} \frac{1}{1 - p^{-z}},
\end{equation}
where $p$ runs over all prime numbers (see \cite{cohen2007number}).
Euler used this product to give a fundamentally new proof of the infinitude of primes (see \cite{davenport2000multiplicative}).

Around 1740, Euler  \cite{euler1740summis}   provided the formula for the values of $\zeta(z)$ at positive even integers
\begin{align}\label{zeta_2n}
    \zeta(2n) = (2\pi)^{2n} \frac{(-1)^n B_{2n}}{2(2n)!},
\end{align}
where $n$ is a positive integer and $B_n$ denotes the $n$-th Bernoulli number, defined by the generating function
\begin{align*}
    \frac{t}{e^t-1} = \sum_{n=0}^{\infty} B_n \frac{t^n}{n!}.
\end{align*}
However, to this day, no explicit closed-form expression analogous to \eqref{zeta_2n} is known for $\zeta(z)$ at odd integers.
In particular, $\zeta(3)$ is known as Apéry's constant. In 1979, at a number theory conference in Marseille, France, Apéry employed a unique construction of continued fractions and recurrence sequences to provide the first rigorous proof that $\zeta(3)$ is irrational, its numerical value being approximately $\zeta(3) \approx 1.2020569$ (see \cite{apery1979irrationalite}). In 2023, Hu and Kim  \cite{hu2023euler} evaluated the definite integral
\begin{align*}
    \int_{0}^{x} \theta^{r-2} \log\left( \cos \frac{\theta}{2} \right) {\rm d}\theta
\end{align*}
for $r=2, 3, 4, \dots$, thereby obtaining the following expressions for $\zeta(3)$:
\begin{align*}
\zeta(3) &= \frac{4\pi^2}{21} \log\left( \frac{e^{\frac{4G}{\pi}} \mathcal{C}_3\left(\frac{1}{4}\right)^{16}}{\sqrt{2}} \right), \\
\zeta(3) &= \frac{72\pi^2}{11} \log\left( \frac{3^{\frac{1}{72}} \mathcal{C}_3\left(\frac{1}{6}\right)}{\mathcal{C}_2\left(\frac{1}{6}\right)^{\frac{1}{3}}} \right),
\end{align*}
where $G$ is Catalan's constant, and $\mathcal{C}_3\left(\frac{1}{4}\right)$ and $\mathcal{C}_3\left(\frac{1}{6}\right)$ denote the special values of Kurokawa and Koyama's triple cosine function at $\frac{1}{4}$ and $\frac{1}{6}$, respectively.

Around 1910, Ramanujan obtained the following identity for $\zeta(2n+1)$ in an unpublished manuscript, which was later recorded by Andrews and Berndt in \cite{AndrewsBerndt2013LostIV}. 
Berndt, from the viewpoints of the general transformation of analytic Eisenstein series, pointed out that this identity is a natural counterpart to the formula for the values of $\zeta(z)$ at even integers, namely Euler's formula \eqref{zeta_2n}.

\begin{theorem}[\textup{Ramanujan's identity, see \cite[pp.~319--320]{AndrewsBerndt2013LostIV}}]
If $\alpha$ and $\beta$ are positive numbers satisfying $\alpha\beta = \pi^2$, and $n$ is a positive integer, then
\begin{equation}\label{Ramanujan_identity}
    \begin{aligned}
            \alpha^{-n}\left\{\frac{1}{2}\zeta(2n+1)
            +\sum_{m=1}^{\infty}\frac{m^{-2n-1}}{e^{2\alpha m}-1}\right\}
            &-(-\beta)^{-n}\left\{\frac{1}{2}\zeta(2n+1)
            +\sum_{m=1}^{\infty}\frac{m^{-2n-1}}{e^{2\beta m}-1}\right\} \\
            &= 2^{2n}\sum_{k=0}^{n+1}\frac{(-1)^n B_{2k}B_{2n-2k+2}}
            {(2k)!(2n-2k+2)!}\,\alpha^{n-k+1}\beta^{k}.
    \end{aligned}
\end{equation}
\end{theorem}

This is the first identity linking $\zeta(2n+1)$ with Bernoulli numbers $B_{n}$ and modular form structures. 
In 1925, Marulkar \cite{Marulkar1925} independently gave another proof of Ramanujan's identity, unaware that the formula had already appeared in Ramanujan's notebooks. Later, Grosswald \cite{Grosswald1972} rediscovered and studied Ramanujan's identity. 
In 1977, Berndt \cite{Berndt1977} showed that both Euler's formula and Ramanujan's identity are special cases of a general transformation formula for analytic Eisenstein series, thereby establishing their companion relationship.
In 2011, Gun, Murty, and Rath \cite{GMR11} gave a new interpretation of \eqref{Ramanujan_identity}, pointing out that it encapsulates the fundamental transformation properties of Eisenstein series on the full modular group and their Eichler integrals.
Subsequently, Berndt and Straub \cite{BerndtStraub2016}, using the method of secant Dirichlet series, generalized such identities to Eisenstein series of level $2$ and weight $2k+1$, establishing a more general transformation formula and revealing connections with generalized Bernoulli numbers and $L$-functions. 
In 2024, as a complement to the work of Berndt and Straub, Dixit \cite{Dixit2024} gave a further survey of recent developments and generalizations of Ramanujan's identity, involving cutting-edge topics such as Lambert series, the Koshliakov zeta function, higher-order Herglotz functions, and non-holomorphic Eisenstein series.
In 2025, Hu and Kim \cite{hu2025analogue} obtained an analogue of Ramanujan's identity in the function field setting, which involves Bernoulli--Carlitz numbers.

In 2022, Chavan \cite{2022Dirichlet} reformulated Ramanujan's identity as an identity concerning convolutions of $\zeta(z)$:

\begin{theorem}[\textup{Ramanujan-type identity for $\zeta(z)$, see \cite[Proposition 2.1]{2022Dirichlet}}]\label{chavan0}
    Let $\alpha$ and $\beta$ be positive numbers such that $\alpha\beta = \pi^2$, and let $n$ be a non-negative integer. Then the equivalent form of Ramanujan's identity is
\begin{equation}
  \begin{aligned}
    &\alpha^{-n}\left\{ \frac{1}{2}\zeta(2n+1) + \sum_{m=1}^{\infty} \frac{m^{-2n-1}}{e^{2\alpha m} - 1} - \frac{1}{2\alpha}\zeta(2n+2) \right\} \\
    &\quad -(-\beta)^{-n}\left\{ \frac{1}{2}\zeta(2n+1) + \sum_{m=1}^{\infty} \frac{m^{-2n-1}}{e^{2\beta m} - 1} - \frac{1}{2\beta}\zeta(2n+2) \right\} \\
    &= \frac{(-1)^n}{\pi^{2n+2}} \sum_{k=1}^{n} \zeta(2k)\zeta(2n-2k+2) \alpha^{n-k+1}\beta^{k}.
  \end{aligned}
  \label{prop:2.1}
\end{equation}
\end{theorem}

In 2023, Chavan  \cite{chavan2023hurwitz} noted that the function $\frac{1}{e^{2\pi x}-1}$ has the following integral representation:
\begin{align*}
    \frac{1}{e^{2\pi x}-1} = \frac{1}{2i\pi} \underset{(c)}{\int} \frac{\zeta(1-s)}
    {2 \cos\left(\frac{\pi s}{2}\right)} x^{-s} \, {\rm d}s,
\end{align*}
where $(c)$ denotes the vertical line $\operatorname{Re}(s)=c$, with $c$ being any real number satisfying $1<c<2$.

This function has simple poles at $x=0$ and $x=\pm in$ ($n\in\mathbb{N}$), and the residue at $0$ and $\pm in$ is $\frac{1}{2\pi}$. Consequently, it admits the following partial fraction expansion:
$$\frac{1}{e^{2\pi x}-1} = -\frac{1}{2} + \frac{1}{2\pi x} + \frac{x}{\pi} \sum_{n=1}^{\infty} \frac{1}{n^2+x^2}$$
(see \cite[p.~2]{chavan2023hurwitz}). Chavan named this function the Ramanujan kernel, and defined its two-parameter generalization as follows:

\begin{definition}[\textup{See \cite[Definition 1.2]{chavan2023hurwitz}}]
    Let $x\in\mathbb{R}^+$, $a\in\mathbb{C}$ and $k\in\mathbb{N}$. The even-order Hurwitz kernel is defined by
    \begin{equation}\label{hurwitz1}
        \begin{aligned}
           \Psi(x,a;k) &= \frac{1}{2i\pi} \int_{(c)} \frac{\zeta(1-s,a)}
           {2k \cos\left(\frac{\pi(s+k-1)}{2k}\right)} x^{-s} \, {\rm d}s \\
           &= \frac{2a-1}{2\pi x} - \frac{1}{2k \cos\left(\frac{\pi(k-1)}{2k}\right)} + \frac{1}{\pi} \sum_{n=0}^{\infty} \frac{x^{2k-1}}{(n+a)^{2k} + x^{2k}},
        \end{aligned}
    \end{equation}
    where $c\in\mathbb{R}$ and $1<c<2$.
\end{definition}

\begin{definition}[\textup{See \cite[Definition 1.4]{chavan2023hurwitz}}]
    Let $x\in\mathbb{R}^+$, $a\in\mathbb{C}$ and $k\in\mathbb{N}$. The odd-order Hurwitz kernel is defined by
    \begin{equation}\label{hurwitz2}
        \begin{aligned}
           \Phi(x,a;k) &= \frac{1}{2i\pi} \int_{(c)} \frac{\zeta(1-s,a)}{2k \sin\left(\frac{\pi s}{2k}\right)} x^{-s} \, {\rm d}s \\
           &= \frac{\log(x) - \psi(a)}{\pi} - \frac{x^{2k}}{\pi} \sum_{n=0}^{\infty} \frac{1}{(n+a)\left((n+a)^{2k} + x^{2k}\right)},
        \end{aligned}
    \end{equation}
    where $c\in\mathbb{R}$ and $1<c<2$.
\end{definition}

For convenience, when dealing with complex integrals appearing in the text, Chavan uses $R_z$ to denote the residue of the integral at the pole $z$, and sets
\begin{align*}
    \Psi\!\left(\frac{\alpha x}{\pi}, a; k\right) = \Psi_\alpha(x,a;k), \qquad
    \Phi\!\left(\frac{\alpha x}{\pi}, a; k\right) = \Phi_\alpha(x,a;k).
\end{align*}

The Hurwitz zeta function $\zeta(z,x)$ is a two variable generalization of the Riemann zeta function (\ref{Rzeta}):
\begin{align*}
    \zeta(z,x) = \sum_{n=0}^{\infty} \frac{1}{(n+x)^z},
\end{align*}
where $x \neq 0, -1, -2, \dots$, $z \in \mathbb{C}$ and $\operatorname{Re}(z) > 1$
(see \cite{hurwitz1882einige}). It is eaily seen that, by setting $x=1$, we recover the Riemann zeta function $\zeta(z)$:
\begin{align*}
    \zeta(z,1) = \zeta(z) = \sum_{n=1}^{\infty} \frac{1}{n^z}.
\end{align*}
In \cite{Dixit2021}, Dixit et al.\ asked: \textbf{Does a similar Ramanujan identity exist for the Hurwitz zeta function?}
Chavan \cite{chavan2023hurwitz} answered this affirmatively. Specifically, he derived the following result:

\begin{theorem}[\textup{Ramanujan-type identity for $\zeta(2k,x)$, see \cite[Theorem 2.1]{chavan2023hurwitz}}]\label{chavan1}
    Let $\alpha$ and $\beta \in \mathbb{R}^+$ satisfy $\alpha\beta = \pi^2$. For $k, N \in \mathbb{N}$, we have
    \begin{equation}\label{chavan1_1}
        \begin{aligned}
            &\beta^{k(N+1)-1} \left( \sum_{n=0}^{\infty}
            \frac{\Psi_\alpha(n+b, a; k)}{(n+b)^{2k(N+1)-1}}
            + \frac{\zeta(2k(N+1)-1, b)}
            {2k \cos\left(\frac{\pi(k-1)}{2k}\right)} \right) \\
            &= (-1)^N \alpha^{k(N+1)-1}
            \left( \sum_{n=0}^{\infty}
            \frac{\Psi_\beta(n+a, b; k)}{(n+a)^{2k(N+1)-1}}
            + \frac{\zeta(2k(N+1)-1, a)}
            {2k \cos\left(\frac{\pi(k-1)}{2k}\right)} \right) \\
            &\quad + \sum_{p=0}^{N+1} (-1)^{p+1} \zeta(2kp, a)
            \zeta(2k(N-p+1), b) \alpha^{kp-1} \beta^{k(N+1-p)-1}.
        \end{aligned}
    \end{equation}
\end{theorem}

Substituting the series expansion of $\Psi_\alpha(x,a;k)$ (see \eqref{hurwitz1}) into \eqref{chavan1_1} yields an alternative form of the identity:
\begin{equation}
    \begin{aligned}
        &\beta^{k(N+1)-1} \left( \sum_{n=0}^{\infty}
        \frac{1}{(n+b)^{2k(N+1)-1}}
        \sum_{i=0}^{\infty} \frac{\alpha^{k-1}(n+b)^{2k-1}}
        {\alpha^k(n+b)^{2k} + \beta^k(i+a)^{2k}} \right) \\
        &= (-1)^N \alpha^{k(N+1)-1}
        \left( \sum_{n=0}^{\infty}
        \frac{1}{(n+a)^{2k(N+1)-1}} \sum_{i=0}^{\infty}
        \frac{\beta^{k-1}(n+a)^{2k-1}}
        {\beta^k(n+a)^{2k} + \alpha^k(i+b)^{2k}} \right) \\
        &\quad + \sum_{p=1}^{N} (-1)^{p+1} \zeta(2kp, a)
        \zeta(2k(N-p+1), b) \alpha^{kp-1} \beta^{k(N+1-p)-1}.
    \end{aligned}
\end{equation}

Building upon Chavan's work, this paper introduces an alternating structure and simultaneously obtains analogues of Theorem~\ref{chavan0} and Theorem~\ref{chavan1} in the setting of the alternating Hurwitz zeta function.

The alternating Hurwitz zeta function is defined by
\begin{equation}\label{alternating_hurwitz}
    \zeta_E(z,x) = \sum_{n=0}^{\infty} \frac{(-1)^n}{(n+x)^z},
\end{equation}
where $x \neq 0, -1, -2, \dots$, $z \in \mathbb{C}$ and $\operatorname{Re}(z) > 0$ (see \cite{choi2011multiple, hu2015special, hu2022stieltjes, hu2024asymptotic, kim2018some}).
The function $\zeta_E(z,x)$ can be analytically continued to the entire complex plane. Setting $x = 1$ yields the alternating zeta function $\zeta_E(z)$, also known as the Dirichlet eta function $\eta(z)$:
\begin{align*}
    \zeta_E(z,1) = \zeta_E(z) = \sum_{n=1}^{\infty} \frac{(-1)^{n-1}}{n^z} = \eta(z),
\end{align*}
where $z \in \mathbb{C}$ and $\operatorname{Re}(z) > 0$ (see \cite{min2014zeros}).
The functions $\zeta_E(z)$ and $\zeta(z)$ satisfy the relation
\begin{align*}
    \zeta_E(z) = \left(1 - \frac{1}{2^{z-1}}\right) \zeta(z).
\end{align*}

In contrast to the Hurwitz zeta function $ \zeta(z,x)$, its alternating counterpart $\zeta_E(z,x)$ possesses certain analytic advantages. This is because $\zeta_E(z,x)$ can be analytically continued to the entire complex plane, whereas $\zeta(z,x)$ has a simple pole at $s = 1$. During the recent years, various analytic properties of  $\zeta_E(z,x)$ have been systematically investigated,  including the Fourier series expansions, power series expansions, asymptotic expansions, integral representations, special values, and convexity (see \cite{cvijovic2020note, hu2015special, hu2024asymptotic, lambda2019hu}). In algebraic number theory, the alternating Hurwitz zeta function $\zeta_E(z,x)$ can be used to express certain partial zeta functions appearing in Stark's conjectures for cyclotomic fields (see \cite[(6.13)]{kim2013some}). 

In this paper, we will extend Ramanujan's identity to $\zeta_E(z,x)$ and systematically study the related modular transformation properties (see Section \ref{Mresults}). To state and prove our main results, in the subsequent section, we need first recall the properties of $\zeta_E(z,x)$ and the related functions.

 \subsection{The alternating Hurwitz zeta and the related functions}
By Williams and Zhang,  $\zeta_E(z, x)$ possesses the following integral representation:

\begin{proposition}[\textup{Integral representation of $\zeta_E(z, x)$, see \cite[(3.1)]{Williams1993}}]
    For $\operatorname{Re}(z) > 0$ and $x > 0$,
    \begin{align*}
        \zeta_E(z, x) = \frac{1}{\Gamma(z)} \int_{0}^{\infty} \frac{e^{-(x-1)t} t^{z-1}}{e^{t} + 1} \, {\rm d}t.
    \end{align*}
\end{proposition}

In 2021, Hu and Kim in \cite{hu2022stieltjes} defined generalized Stieltjes constants via the Taylor series expansion of $\zeta_E(z,x)$ at $z = 1$:
\begin{equation}\label{1-35}
\begin{aligned}
    \zeta_E(z,x) = \sum_{k=0}^{\infty} \frac{(-1)^k \widetilde{\gamma}_k(x)}{k!} (z-1)^k.
\end{aligned}
\end{equation}
When $x = 1$, we denote $\widetilde{\gamma}_k = \widetilde{\gamma}_k(1)$. These constants are defined by the Taylor expansion of $\zeta_E(z)$ at $z = 1$:
\begin{align*}
    \zeta_E(z) = \sum_{k=0}^{\infty} \frac{(-1)^k \widetilde{\gamma}_k}{k!} (z-1)^k.
\end{align*}
By definition of $\zeta_E(z,x)$, it is clear that
\begin{align*}
    \widetilde{\gamma}_0(x) = \zeta_E(1,x).
\end{align*}
Using the Boole summation formula, they also proved the following proposition:

\begin{proposition}[\textup{See \cite[Proposition 3.1]{hu2022stieltjes}}]
Let $x > 0$, $\operatorname{Re}(z) > -1$ and $\alpha = 0, 1, 2, \dots$. Then
\begin{align*}
\zeta_E(z, x) = \sum_{n=0}^{\alpha} \frac{(-1)^n}{(n+x)^z}
- \frac{1}{2} \cdot \frac{(-1)^\alpha}{(\alpha+x)^z}
- \frac{z}{2} \int_{\alpha}^{\infty} \frac{\overline{E}_0(-t)}{(t+x)^{z+1}} \, {\rm d}t,
\end{align*}
where $\overline{E}_0(t)$ is the $0$-th periodic Euler function, whose Fourier series expansion is
\begin{align*}
\overline{E}_0(t) = \frac{4}{\pi} \sum_{k=0}^{\infty} \frac{\sin\big((2k+1)\pi t\big)}{2k+1},
\end{align*}
for $t$ not an integer.
\end{proposition}

In particular, setting $\alpha = 0$ yields

\begin{corollary}[\textup{See \cite[Corollary 3.3]{hu2022stieltjes}}]
    For $\operatorname{Re}(z) > -1$,
\begin{align*}
\zeta_E(z, x) = \frac{1}{2 x^z} - \frac{1}{2} z \int_{0}^{\infty} 
\frac{\overline{E}_0(-t)}{(t + x)^{z+1}} \, {\rm d}t.
\end{align*}
Setting $x = 1$ gives an expression for the alternating zeta function $\zeta_E(z)$:
\begin{align*}
\zeta_E(z) = \zeta_E(z, 1) = \frac{1}{2} + \frac{1}{2} z \int_{1}^{\infty} 
\frac{\overline{E}_0(-t)}{t^{z+1}} \, {\rm d}t.
\end{align*}
\end{corollary}

The function $\zeta_E(z,x)$ shares many properties with $\zeta(z,x)$. 
Let $z, \delta > 0$ and $|\arg(x)| \le \pi - \delta$. As $|x| \to \infty$, the asymptotic expansion of $\zeta_E(z,x)$ is
\begin{align*}
    \zeta_E(z,x) = \frac{1}{2}x^{-z} + \frac{1}{4}z x^{-z-1}
    - \frac{1}{2}x^{-z} \sum_{k=1}^{\infty} \frac{E_{2k+1}(0)}{(2k+1)!} \frac{(z)_{2k+1}}{x^{2k+1}},
\end{align*}
where $E_{2k+1}(0)$ denotes the values of the odd-index Euler polynomials $E_n(x)$ (see \cite{artin1964gamma}) at $0$, and $(z)_{2k+1}$ is the rising factorial defined by
\begin{align*}
    (z)_k = z(z+1)\cdots(z+k-1) = \frac{\Gamma(z+k)}{\Gamma(z)}
\end{align*}
(see \cite[Theorem 3.1]{hu2024asymptotic}).

For $x>0$,  Hu and Kim \cite{hu2022stieltjes} defines $\widetilde{\psi}(x)$ as
\begin{equation}
    \widetilde{\psi}(x):=-{\widetilde{\gamma}}_0\left(x\right)=-\zeta_E\left(1,x\right),
\end{equation}
which can also be defined as
    \begin{align*}
        \widetilde{\psi}\left(x\right)&=-\frac{\Gamma^\prime\left(x\right)}{\Gamma\left(x\right)}+\frac{\Gamma^\prime\left(\frac{x}{2}\right)}{\Gamma\left(\frac{x}{2}\right)}+\log{2}
        \\&=\psi\left(x\right)+\psi\left(\frac{x}{2}\right)+\log{2}.
    \end{align*}
Moreover, it has the following integral representations:
\begin{align*}
\widetilde{\psi}(x) &= \widetilde{\gamma}_0 + \int_0^\infty \frac{-e^{-t} - e^{-xt}}{1 + e^{-t}} {\rm d}t \\
&= \widetilde{\gamma}_0 - \int_0^1 \frac{1 + t^{x-1}}{1 + t} {\rm d}t \\
&= -\int_0^\infty \frac{e^{-xt}}{1 + e^{-t}} {\rm d}t,
\end{align*}
where $\tilde{\gamma}_0 = \tilde{\gamma}_0(1) = \zeta_E(1,1)$.

Then Hu and Kim \cite{hu2022stieltjes} obtains the special values of $\widetilde{\psi}(x)$ at positive integers:

\begin{proposition}[\textup{Special values of $\widetilde{\psi}(x)$ at positive integers, see \cite[Corollary 3.14]{hu2022stieltjes} and \cite[Corollary 5]{wang2026gamma}}]
   For $n\in\mathbb{N}$, we have
$$
(-1)^{n}\widetilde{\psi}(n)=
\widetilde{\gamma}_{0}+\displaystyle\sum_{k=1}^{n-1}\frac{(-1)^{k}}{k}, $$
where $\t\gamma_{0}$ is the Euler constant respect to $\zeta_{E}(z,x)$.
\end{proposition}

In 2024, Wang, Hu and Kim in \cite{wang2026gamma} further investigated the properties of the Gamma function $\widetilde{\Gamma}(x)$ corresponding to $\zeta_E(z,x)$. They studied its integral representation, recurrence relations, and other characteristics, and obtained a series of fundamental properties for the alternating digamma function $\widetilde{\psi}(x)$, such as recurrence formulas and reflection formulas.

\begin{theorem}[\textup{Definition of $\widetilde{\Gamma}(x)$, see \cite[Theorem 3.12]{hu2022stieltjes}}]
If we define the Gamma function $\widetilde{\Gamma}(x)$ associated with $\zeta_E(z,x)$ by the differential equation
\begin{align*}
    \widetilde{\psi}(x) = \frac{\rm d}{{\rm d} x} \log \widetilde{\Gamma}(x),
\end{align*}
then it admits the following infinite product expansion:
\begin{align*}
\widetilde{\Gamma}(x) = \frac{1}{x} e^{\widetilde{\gamma}_0 x} \prod_{k=1}^{\infty} 
\left( e^{-\frac{x}{k}} \left( 1 + \frac{x}{k} \right) \right)^{(-1)^{k+1}}.
\end{align*}
Moreover, we have
\begin{equation}\label{digamma1_0}
\widetilde{\psi}(x) = -\frac{1}{x} + \widetilde{\gamma}_0 + \sum_{k=1}^{\infty} (-1)^k \left( \frac{1}{k} - \frac{1}{k+x} \right),
\end{equation}
where $\widetilde{\gamma}_0 = \widetilde{\gamma}_0(1) = \zeta_E(1,1)$.
\end{theorem}

\begin{proposition}[\textup{Integral representation of $\widetilde{\Gamma}(x)$, see \cite[Theorem 1]{wang2026gamma}}]
    For $\operatorname{Re}(x) > 0$
    \begin{align*}
        \widetilde{\Gamma}(x) = \int_{0}^{\infty} e^{-xt} \left( 1 - e^{-2t} \right)^{-1/2} \, {\rm d}t.
    \end{align*}
\end{proposition}

\begin{proposition}[\textup{Recurrence relation for $\widetilde{\Gamma}(x)$, see \cite[Theorem 3]{wang2026gamma}}]
    For $\operatorname{Re}(x) > 0$ and $n \in \mathbb{N}$, the recurrence relation for $\widetilde{\Gamma}(x)$ is
    \begin{align*}
        \left( \widetilde{\Gamma}(x + n) \right)^{(-1)^n} = \widetilde{\Gamma}(x) \prod_{k=0}^{n-1} \left( \frac{2(x + k)}{\pi} \right)^{(-1)^k}.
    \end{align*}
    In particular, for $n = 1$
    \begin{align*}
         \widetilde{\Gamma}(x + 1) \widetilde{\Gamma}(x) = \frac{\pi}{2 x}.
    \end{align*}
\end{proposition}

\begin{proposition}[\textup{Reflection formula for $\widetilde{\Gamma}(x)$, see \cite[Theorem 7]{wang2026gamma}}]
    For $0 < \operatorname{Re}(x) < 1$,
    \begin{align*}
        \frac{\widetilde{\Gamma}(x)}{\widetilde{\Gamma}(1 - x)} = \cot\!\left( \frac{\pi x}{2} \right).
    \end{align*}
\end{proposition}

\begin{proposition}[\textup{Recurrence relation  and reflection formula for $\widetilde{\psi}(x)$, see  \cite[Theorem 8]{wang2026gamma}}]
  For {\rm Re}$(x)>0$ and $n\in\mathbb{N}$, we have the recursive formula
	\begin{equation}\label{1.22-1}
		(-1)^{n}\,\widetilde{\psi}\left(x+n\right)=\widetilde{\psi}\left(x\right)+\sum_{k=0}^{n-1}\frac{\left(-1\right)^{k}}{x+k}\end{equation}
	and for $0<{\rm Re}(x)<1$, we have the reflection equation 
	\begin{equation}\label{1.22-2}
		\widetilde{\psi}(x)+\widetilde{\psi}(1-x)=-\frac{\pi}{\sin\pi x}.
	\end{equation}
	\end{proposition}

In  2025, building upon the works by Z.-W. Sun \cite{Sun1989congruences, Sun2001Algebraic, Sun2002covering} 
and Z.-H. Sun \cite{Sun2023invariant} on invariant functions, Zhu, Hu and Kim \cite{zhu2025study} introduced the concept of alternating 
invariant functions, defined as two-variable functions $f(x,y)$ satisfying the functional equation
\begin{align*}
			\sum_{r=0}^{n-1} (-1)^r f(x+ry, ny) = f(x,y), \quad \text{for any positive odd integer $n$}.
					\end{align*} 
This class of functions is closed under translation, reflection, and differentiation. Concrete examples include the Euler polynomials $E_n(x)$, the alternating Hurwitz zeta function $\zeta_E(z,x)$ and its associated Gamma function $\widetilde{\Gamma}(x)$, the alternating digamma function $\widetilde{\psi}(x)$, among others.

\subsection{Main Results}\label{Mresults}
The main aim of this paper is to establish and study the Ramanujan-type identities for  $\zeta_E(z,x)$:

\begin{theorem}[\textup{Ramanujan-type identity for $\zeta_E(2k,x)$}]\label{res1}
    Let $\alpha, \beta \in \mathbb{C}$ satisfy $\operatorname{Re}(\alpha) > 0$, $\operatorname{Re}(\beta) > 0$ and $\alpha\beta = 4\pi^2$. Then for $N \in \mathbb{N}$, we have
    \begin{equation}
        \begin{aligned}
            &\sum_{k=1}^{N} (-\alpha)^{N+1-k} \beta^{k} \, \zeta_E(2k,x) \, \zeta_E(2(N+1-k),y) \\
            &= \beta^{N+1} \frac{1}{2} i \sqrt{\frac{\alpha}{\beta}}
            \sum_{n=0}^{\infty} \frac{(-1)^n}{(n+x)^{2N+1}}
            \left[ \widetilde{\psi}\!\left(y + i\sqrt{\frac{\alpha}{\beta}}(n+x)\right)
            - \widetilde{\psi}\!\left(y - i\sqrt{\frac{\alpha}{\beta}}(n+x)\right) \right] \\
            &\quad + (-\alpha)^{N+1} \frac{1}{2} i \sqrt{\frac{\beta}{\alpha}}
            \sum_{n=0}^{\infty} \frac{(-1)^n}{(n+y)^{2N+1}}
            \left[ \widetilde{\psi}\!\left(x + i\sqrt{\frac{\beta}{\alpha}}(n+y)\right)
            - \widetilde{\psi}\!\left(x - i\sqrt{\frac{\beta}{\alpha}}(n+y)\right) \right].
        \end{aligned}
    \end{equation}
\end{theorem}

\begin{theorem}[\textup{Ramanujan-type identity for $\zeta_E(2k+1)$}]\label{res2}
    Let $\alpha, \beta \in \mathbb{C}$ satisfy $\operatorname{Re}(\alpha) > 0$, $\operatorname{Re}(\beta) > 0$ and $\alpha\beta = 4\pi^2$. Then for $m \in \mathbb{N}$, we have
    \begin{equation}
    \begin{aligned}
        &(-\beta)^{-m} \left\{ 2\widetilde{\gamma}_0 \zeta_E(2m+1) + \sum_{n=1}^{\infty}
        \frac{(-1)^n}{n^{2m+1}} \left( \widetilde{\psi}\!\left(\frac{in\alpha}{2\pi}\right)
        + \widetilde{\psi}\!\left(-\frac{in\alpha}{2\pi}\right) \right) \right\} \\
        &\quad + \alpha^{-m} \left\{ 2\widetilde{\gamma}_0 \zeta_E(2m+1) + \sum_{n=1}^{\infty}
        \frac{(-1)^n}{n^{2m+1}} \left( \widetilde{\psi}\!\left(\frac{in\beta}{2\pi}\right)
        + \widetilde{\psi}\!\left(-\frac{in\beta}{2\pi}\right) \right) \right\} \\
        &= -2 \sum_{k=1}^{m-1} (-\beta)^{k-m} \alpha^{-k} \zeta_E(2k+1) \, \zeta_E(2m-2k+1).
    \end{aligned}
    \end{equation}
\end{theorem}

Under the modular symmetry condition $\alpha\beta = \frac{\pi^2}{4}$, we establish an infinite series expression for the product of the tangent and hyperbolic tangent functions:

\begin{theorem}[\textup{(Infinite series form for the product of tangent and hyperbolic tangent}]\label{res3}
    Let $\alpha, \beta \in \mathbb{R}^+$ satisfy $\alpha\beta = \frac{\pi^2}{4}$. If $\omega \neq -(2m+1)^2\alpha$ and $\omega \neq (2m+1)^2\beta$ for $0 \le m < \infty$, then the following formula holds
\begin{equation}\label{eq8}
\begin{aligned}
&\frac{\pi}{4} \tan\!\left(\sqrt{\omega\alpha}\right) \tanh\!\left(\sqrt{\omega\beta}\right) \\
&= \sum_{m=0}^{\infty} \left\{ \frac{(2m+1)\beta \tanh\!\big((2m+1)\beta\big)}
{\omega - (2m+1)^2\beta} - \frac{(2m+1)\alpha \tanh\!\big((2m+1)\alpha\big)}
{\omega + (2m+1)^2\alpha} \right\}.
\end{aligned}
\end{equation}
\end{theorem}

The Dirichlet lambda function is defined by
\begin{align}
    \lambda(z) = \sum_{n=0}^\infty \frac{1}{(2n+1)^z}, \quad \operatorname{Re}(z) > 1.
\end{align}
For real variables, it was studied by Euler (see the exposition by Varadarajan \cite[p.~70]{Varadarajan2006}), while the complex variable case was systematically investigated by Dirichlet. For historical remarks and related results on $\lambda(z)$, see the introduction of \cite{lambda2019hu}.

Under the modular symmetry condition $\alpha\beta = \frac{\pi^2}{4}$, we express the Dirichlet lambda function $\lambda(z)$ together with linear combinations of infinite series as convolution sums of Bernoulli numbers $B_{n}$, special values of Euler polynomials at $0$, $E_n(0)$, and Genocchi numbers $G_{n}$.

\begin{theorem}[\textup{Ramanujan-type identity for $\lambda(z)$}]\label{res4}
Let $\alpha, \beta \in \mathbb{R}^+$ satisfy $\alpha\beta = \frac{\pi^2}{4}$. Suppose $\omega \neq -(2m+1)^2\alpha$, $\omega \neq (2m+1)^2\beta$ for $0 \le m < \infty$. Let $B_n$ be the Bernoulli numbers, $E_n(0)$ the special values of Euler polynomials at $0$, $G_n$ the Genocchi numbers and $\lambda(z)$ the Dirichlet lambda function. Then the following formula holds
  \begin{equation}\label{Th1}
    \begin{aligned}
    	&-2\beta^{-r}\left(\frac{1}{2}\lambda\left(2r+1\right)-\sum_{m=0}^{\infty}
    	\frac{\left(2m+1\right)^{-2r-1}}{e^{2\left(2m+1\right)\beta}+1}\right)
    	\\
    	&\qquad+2(-1)^r\alpha^{-r}\left(\frac{1}{2}\lambda\left(2r+1\right)-
    	\sum_{m=0}^{\infty}\frac{\left(2m+1\right)^{-2r-1}}{e^{2\left(2m+1\right)\alpha}+1}\right)\\
        &=-\frac{1}{2}\sum_{k=1}^{r}{\frac{2^{2r+2}\left(2^{2k}-1\right)
        		\left(2^{2\left(r+1-k\right)}-1\right)\left|B_{2k}\right|B_{2\left(r+1-k\right)}}
        	{\left(2k\right)!\left(2\left(r+1-k\right)\right)!}\alpha^k\beta^{r+1-k}}
        \\&=-2^{2r+1}\sum_{k=1}^{r}{\frac{k(r+1-k)E_{2k-1}(0)E_{2(r+1-k)-1}(0)}
        	{\left(2k\right)!\left(2\left(r+1-k\right)\right)!}\alpha^k\beta^{r+1-k}}
        \\&=-2^{2r-1}\sum_{k=1}^{r}{\frac{G_{2k}G_{2\left(r+1-k\right)}}
        	{\left(2k\right)!\left(2\left(r+1-k\right)\right)!}\alpha^k\beta^{r+1-k}}.
    \end{aligned}
  \end{equation}
\end{theorem}

We now present the definitions of the even-order and odd-order alternating Hurwitz kernels, including both their integral and series forms:

\begin{definition}\label{alternating_hurwitz_kernel_even}
    Let $x \in \mathbb{R}^+$, $a \in \mathbb{C}$, $\operatorname{Re}(1-s) > 0$ and $k \in \mathbb{N}$.
    The even-order alternating Hurwitz kernel is defined by
        \begin{equation}\label{Def1.19}
        \begin{aligned}
            F(x,a;k) &= \frac{1}{2i\pi} \int_{(c)}
            \frac{\zeta_E(1-s,a)}
            {2k \cos\!\left(\frac{\pi(s+k-1)}{2k}\right)} \, x^{-s} \, {\rm d}s \\
            &= -\frac{1}{2\pi x} + \frac{x^{2k-1}}{\pi} \sum_{n=0}^{\infty}
            \frac{(-1)^n}{(n+a)^{2k} + x^{2k}},
        \end{aligned}
        \end{equation}
        where $c \in \mathbb{R}$ and $1 < c < 2$.
\end{definition}

\begin{definition}\label{alternating_hurwitz_kernel_odd}
    Let $x \in \mathbb{R}^+$, $a \in \mathbb{C}$, $\operatorname{Re}(1-s) > 0$ and $k \in \mathbb{N}$.
    The odd-order alternating Hurwitz kernel is defined by
    \begin{equation}\label{Def1.20}
        \begin{aligned}
            G(x,a;k) &= \frac{1}{2i\pi} \int_{(c)}
            \frac{\zeta_E(1-s,a)}{2k \sin\!\left(\frac{\pi s}{2k}\right)} \, x^{-s} \, {\rm d}s \\
            &= -\frac{\widetilde{\psi}(a)}{\pi} - \frac{x^{2k}}
            {\pi} \sum_{n=0}^{\infty} \frac{(-1)^n}
            {(n+a)\big((n+a)^{2k} + x^{2k}\big)},
        \end{aligned}
    \end{equation}
    where $c \in \mathbb{R}$ and $1 < c < 2$.
\end{definition}

Based on the above definitions of alternating Hurwitz kernels, we first derive an extension of Theorem~\ref{chavan1} to the alternating setting.
Then, by substituting the series expansion of the even-order alternating Hurwitz kernel into \eqref{gen1}, we obtain an equivalent double alternating series representation as shown in \eqref{gen1.1}.
For convenience, we denote by $R_z$ the residue of the integrand at the pole $z$, and set
\begin{equation}
    F\!\left(\frac{\alpha x}{\pi}, a; k\right) = F_\alpha(x,a;k), \qquad
    G\!\left(\frac{\alpha x}{\pi}, a; k\right) = G_\alpha(x,a;k).
\end{equation}

\begin{theorem}[\textup{Ramanujan-type identity for $F(x,a;k)$}]\label{res5}
    Let $\alpha, \beta \in \mathbb{R}^+$ satisfy $\alpha\beta = \pi^2$ and let $k, N \in \mathbb{N}$. Then the following formula holds
\begin{equation}\label{gen1}
    \begin{aligned}
    &\beta^{k(N+1)-1} \left( \sum_{n=0}^{\infty}
    \frac{(-1)^n F_\alpha(n+b, a; k)}
    {(n+a)^{2k(N+1)-1}} \right) \\
    &= (-1)^N \alpha^{k(N+1)-1}
    \left( \sum_{n=0}^{\infty} \frac{(-1)^n
    F_\beta(n+b, a; k)}{(n+a)^{2k(N+1)-1}} \right) \\
    &\quad + \sum_{p=0}^{N+1} (-1)^{p+1} \zeta_E(2kp, a)
    \, \zeta_E(2k(N-p+1), b) \, \alpha^{kp-1} \beta^{k(N+1-p)-1}.
    \end{aligned}
\end{equation}
\end{theorem}

Substituting the series expressions for $F_\alpha(n+b, a; k)$ and $F_\beta(n+b, a; k)$ (see Definition~\ref{alternating_hurwitz_kernel_even}) into \eqref{gen1} yields the following alternative form of \eqref{gen1}:
\begin{equation}\label{gen1.1}
    \begin{aligned}
        &\beta^{k(N+1)-1}
        \left( \sum_{n=0}^{\infty} \frac{(-1)^n \alpha^{k-1}}
        {(n+b)^{2k(N+1)-1}} \sum_{i=0}^{\infty}
        \frac{(-1)^i (n+b)^{2k-1}}{\alpha^k (n+b)^{2k} + \beta^k (i+a)^{2k}} \right) \\
        &= (-1)^N \alpha^{k(N+1)-1}
        \left( \sum_{n=0}^{\infty} \frac{(-1)^n \beta^{k-1}}
        {(n+a)^{2k(N+1)-1}} \sum_{i=0}^{\infty}
        \frac{(-1)^i (n+a)^{2k-1}}{\beta^k (n+a)^{2k} + \alpha^k (i+b)^{2k}} \right) \\
        &\quad + \sum_{p=1}^{N} (-1)^{p+1}
        \zeta_E(2kp, a) \, \zeta_E(2k(N-p+1), b) \, \alpha^{kp-1} \beta^{k(N+1-p)-1}.
    \end{aligned}
\end{equation}

\begin{corollary}\label{co-3-1}
	Let $\alpha\beta=\pi^2$, $a=b$ in \eqref{gen1}
	and $k,N\in\mathbb{N}$. Then we have
	\begin{align*}
		& \beta^{k(N+1)-1} \left( \sum_{n=0}^{\infty} \frac{(-1)^n F_\alpha(n+a,a;k)}{(n+a)^{2k(N+1)-1}} \right) \\
		&= (-1)^N \alpha^{k(N+1)-1} \left( \sum_{n=0}^{\infty} \frac{(-1)^n F_\beta(n+a,a;k)}{(n+a)^{2k(N+1)-1}} \right) \\
		&\quad+ \sum_{p=0}^{N+1} (-1)^{p+1} \zeta(2kp,a)\zeta(2k(N-p+1),a)\alpha^{kp-1}\beta^{k(N+1-p)-1}.
	\end{align*}
\end{corollary}

\begin{corollary}
	Let $\alpha\beta=\pi^2$ and $a=1$ in Corollary \ref{co-3-1}. Then the following formula holds:
	\begin{equation}
		\begin{aligned}
			&\beta^{k\left(N+1\right)-1}\left(\sum_{n=1}^{\infty}
			\frac{\left(-1\right)^nF_\alpha\left(n,1;k\right)}
			{n^{2k\left(N+1\right)-1}}\right)
			\\&=\left(-1\right)^N\alpha^{k\left(N+1\right)-1}
			\left(\sum_{n=1}^{\infty}\frac{\left(-1\right)^n
				F_\beta\left(n,1;k\right)}{n^{2k\left(N+1\right)-1}}\right)
			\\
			&\quad+\sum_{p=0}^{N+1}{\left(-1\right)^{p+1}
				\zeta_E\left(2kp,a\right)\zeta_E\left(2k(N-p+1),a\right)
				\alpha^{kp-1}\beta^{k(N+1-p)-1}}.
		\end{aligned}
	\end{equation}
\end{corollary}

\begin{corollary}
	When $N=2q+1$, $q\in\mathbb{N}$ and $\alpha=\beta=\pi$ in Theorem \ref{res5}, we have
	\begin{equation}\label{co-l-2}
		\begin{aligned}
			&\sum_{n=0}^{\infty}\frac{(-1)^n}{(n+b)^{4k(q+1)}}
			\sum_{i=0}^{\infty}\frac{(-1)^i (n+b)^{2k}}{(n+b)^{2k} + (i+a)^{2k}}\\
			&\quad+ \sum_{n=0}^{\infty}\frac{(-1)^n}{(n+a)^{4k(q+1)}}
			\sum_{i=0}^{\infty}\frac{(-1)^i (n+a)^{2k}}{(n+a)^{2k} + (i+b)^{2k}}\\
			&= \sum_{p=1}^{2q+1} (-1)^{p+1} \zeta_E(2kp,a) \zeta_E(2k(N-p+1),b).
		\end{aligned}
	\end{equation}
\end{corollary}

\begin{corollary}
	When $a=b$ in (\ref{co-l-2}), we obtain
	\begin{equation}
		\begin{aligned}
			&\sum_{n=0}^{\infty}\frac{(-1)^n}{(n+a)^{4k(q+1)}}
			\sum_{i=0}^{\infty}\frac{(-1)^i (n+a)^{2k}}{(n+a)^{2k} + (i+a)^{2k}}\\
			&= \frac{1}{2}\sum_{p=1}^{2q+1} (-1)^{p+1}
			\zeta_E(2kp,a)\zeta_E\bigl(2k(2q-p+2),a\bigr).
		\end{aligned}
	\end{equation}
\end{corollary}

\begin{theorem}[\textup{Ramanujan-type identity for $G(x,a;k)$}]\label{res6}
    Let $\alpha, \beta \in \mathbb{R}^+$ satisfy $\alpha\beta = \pi^2$ and let $k, m \in \mathbb{N}$. Then the following formula holds:
    \begin{equation}\label{ra-ty-G}
    \begin{aligned}
    &\beta^{km} \sum_{n=0}^{\infty} \frac{(-1)^n G_\alpha(n+b, a; k)}{(n+b)^{2km+1}} \\
    &= (-1)^{m+1} \alpha^{km} \sum_{n=0}^{\infty}
    \frac{(-1)^n G_\beta(n+a, b; k)}{(n+a)^{2km+1}}  \\
    &\quad + \frac{1}{\pi} \sum_{p=0}^{m} (-1)^p \zeta_E(2kp+1, a) \, \zeta_E(2k(m-p)+1, b) \, \alpha^{kp} \beta^{k(m-p)}.
    \end{aligned}
    \end{equation}
    The equation (\ref{ra-ty-G}) can be rewritten as 
    \begin{equation}\label{Th3}
    	\begin{aligned}
    		&(-1)^m\alpha^{km}\sum_{n=0}^{\infty}
    		\frac{\left(-1\right)^n\beta^k}{\left(n+a\right)^{2k(m-1)+1}}
    		\sum_{i=0}^{\infty}\frac{\left(-1\right)^i}{\left(i+b\right)
    			\left(\alpha^k\left(i+b\right)^{2k}+\beta^k\left(n+a\right)^{2k}\right)}
    		\\&=\beta^{km}\sum_{n=0}^{\infty}
    		\frac{\left(-1\right)^n\alpha^k}{\left(n+b\right)^{2k(m-1)+1}}
    		\sum_{i=0}^{\infty}\frac{\left(-1\right)^i}{\left(i+a\right)
    			\left(\beta^k\left(i+a\right)^{2k}+\alpha^k\left(n+b\right)^{2k}\right)}
    		\\
		&\quad+\sum_{p=1}^{m-1}{\left(-1\right)^p\zeta_E\left(2kp+1,a\right)
    			\zeta_E\left(2k\left(m-p\right)+1,b\right)\alpha^{kp}\beta^{k\left(m-p\right)}}.
    	\end{aligned}
    \end{equation}
\end{theorem}

Under the modular symmetry condition $\alpha\beta = \pi^2$, we also obtain Ramanujan-type identities involving Euler polynomials for both the odd-order and even-order alternating Hurwitz kernels.

\begin{theorem}[\textup{Ramanujan-type identity for $F(x,a;k)$ involving Euler polynomials}]\label{res7}
    Let $\alpha, \beta \in \mathbb{R}^+$ satisfy $\alpha\beta = \pi^2$a nd let $E_n(x)$ denote the Euler polynomials. Then the following formula holds
    \begin{equation}\label{Th4}
    \begin{aligned}
       &\alpha^{km+1} \sum_{n=0}^{\infty}
        (-1)^n (n+b)^{2km+1} \left( F_\alpha(n+b, a; k)
        - \sum_{p=1}^{m} \frac{(-1)^n E_{2kp}(a)}{2\pi} \left(\frac{\pi}{\alpha}\right)^{2kp+1} \right) \\
        &= (-1)^{m+1} \beta^{km+1} \sum_{n=0}^{\infty}
        (-1)^n (n+a)^{2km+1} \left( F_\beta(n+a, b; k) -
        \sum_{p=1}^{m} \frac{(-1)^n E_{2kp}(b)}{2\pi}
        \left(\frac{\pi}{\beta}\right)^{2kp+1} \right) \\
        &\quad + \frac{1}{4} \sum_{p=0}^{m} (-1)^{p+1} E_{2kp}(a) E_{2k(m-p)}(b) \, \alpha^{k(m-p)} \beta^{kp}.
    \end{aligned}
    \end{equation}    
\end{theorem}

\begin{theorem}[\textup{Ramanujan-type identity for $G(x,a;k)$ involving Euler polynomials}]\label{res8}
    Let $\alpha, \beta \in \mathbb{R}^+$ satisfy $\alpha\beta = \pi^2$ and let $E_n(x)$ denote the Euler polynomials. Then the following formula holds
    \begin{equation}\label{Th5}
    \begin{aligned}
        &\alpha^{km} \sum_{n=0}^{\infty} (-1)^n (n+b)^{2km-1}
        \left( G_\alpha(n+b, a; k) - \frac{1}{2\pi} \sum_{p=0}^{m} (-1)^n E_{2kp-1}(a) \left(\frac{\pi}{\alpha}\right)^{-2kp} \right) \\
        &= (-1)^{m+1} \beta^{km}
        \sum_{n=0}^{\infty} (-1)^n (n+a)^{2km-1}
        \left( G_\beta(n+a, b; k) - \frac{1}{2\pi} \sum_{p=0}^{m}
        (-1)^n E_{2kp-1}(b) \left(\frac{\pi}{\beta}\right)^{-2kp} \right) \\
        &\quad + \frac{1}{4\pi} \sum_{p=0}^{m} (-1)^p E_{2kp-1}(a) E_{2k(m-p)-1}(b) \, \alpha^{k(m-p)} \beta^{kp},
    \end{aligned}
\end{equation}
where $E_{-1}(a)$ is interpreted so that
$$
E_{-1}(a)=2\zeta_E(1,a).
$$
\end{theorem}

Finally, under the modular symmetry condition $\alpha\beta = \pi^2$, we derive a transformation formula relating the odd-order and even-order alternating Hurwitz kernels:

\begin{theorem}[\textup{Transformation formula between $F(x,a;k)$ and $G(x,a;k)$}]\label{res9}
    Let $\alpha, \beta \in \mathbb{R}^+$ satisfy $\alpha\beta = \pi^2$. Then the following formula holds
    \begin{equation}\label{ra-ty-G2}
    \begin{aligned}
        &\pi \beta^{km} \sum_{n=0}^{\infty}
        \frac{(-1)^n G_\alpha(n+b, a; k)}{(n+b)^{2km}} \\
        &= (-1)^m \pi^2 \alpha^{km+1}
        \sum_{n=0}^{\infty} \frac{(-1)^n F_\beta(n+a, b; k)}{(n+a)^{2km}} \\
        &\quad + \sum_{p=1}^{m} (-1)^p \zeta_E(1+2kp, a) \, \zeta_E(2km-2kp, b) \, \alpha^{kp} \beta^{k(m-p)}.
    \end{aligned}
    \end{equation}
 The equation (\ref{ra-ty-G2}) can be rewritten as 
 \begin{equation}\label{Th6}
 	\begin{aligned}
 		&(-1)^{m-1}\alpha^{km}\sum_{n=0}^{\infty}
 		\frac{\left(-1\right)^n\beta^k}{\left(n+a\right)^{2k(m-1)+1}}
 		\sum_{i=0}^{\infty}\frac{\left(-1\right)^i}{
 			\alpha^k\left(i+b\right)^{2k}+\beta^k\left(n+a\right)^{2k}}
 		\\&=-\beta^{km}\sum_{n=0}^{\infty}
 		\frac{\left(-1\right)^n\alpha^k}{\left(n+b\right)^{2k(m-1)}}
 		\sum_{i=0}^{\infty}\frac{\left(-1\right)^i}{\left(i+a\right)
 			\left(\beta^k\left(i+a\right)^{2k}+\alpha^k\left(n+b\right)^{2k}\right)}
 		\\
		&\quad+\sum_{p=1}^{m-1}{\left(-1\right)^p\zeta_E\left(2kp+1,a\right)
 			\zeta_E\left(2k\left(m-p\right)+1,b\right)\alpha^{kp}\beta^{k\left(m-p\right)}}.
 	\end{aligned}
 \end{equation}
\end{theorem}

The remainder of this paper will be organized as follows. In Section 2, we briefly introduce the necessary preliminary knowledge, focusing on the definitions and convolution operations of Dirichlet series and the Dirichlet lambda function $\lambda(z)$. In Section 3, we prove Theorems \ref{res1} and \ref{res2}, which are Ramanujan-type identities for $\zeta_{E}(z,x)$. Then in Section 4, we prove Theorems \ref{res3} and \ref{res4}, which give an infinite series representation for the product of the tangent and hyperbolic tangent functions and a Ramanujan-type identity for the Dirichlet lambda function $\lambda(z)$, respectively. In Section 5, under the modular symmetry condition $\alpha\beta = \pi^2$, we first present the definitions of the alternating even-order and odd-order Hurwitz kernels (see Definitions \ref{alternating_hurwitz_kernel_even} and \ref{alternating_hurwitz_kernel_odd}). Then, using Cauchy's residue theorem and the theory of Dirichlet series, we derive a series of convolution identities involving the alternating Hurwitz zeta function (see Theorems \ref{res5}--\ref{res9}).

\section{Preliminaries}
In this section, we briefly introduce the preliminary knowledge required for the paper, focusing on the definition and convolution operations of Dirichlet series and the Dirichlet lambda functions $\lambda(z)$.
\subsection{Dirichlet Series}
\begin{definition}[\textup{See \cite[(2.1)]{2022Dirichlet}}]
    For a sequence of non-zero complex numbers ${\{x_n\}}$ and an associated sequence of complex numbers ${\{a_n\}}$, the \textup{Dirichlet} series is defined as
    \begin{equation}\label{dirichletzeta}
        \zeta_{x,a}\left(N\right)=\sum_{n=1}^{\infty}\frac{a_n}{{x_n}^N}.
    \end{equation}
\end{definition}
\begin{remark}
    It is assumed that the \textup{Dirichlet} series converges for $N\geq1$. If the series diverges at $N=1$, but has a finite abscissa of convergence, similar results can still be obtained.
\end{remark}
\begin{definition}[\textup{Zeta generating function, see \cite[(2.2)]{2022Dirichlet}}]
    For a given \textup{Dirichlet} series, the corresponding \textup{zeta} generating function is
    \begin{equation}\label{zetasc}
        \psi_{x,a}\left(z\right)=\sum_{N=1}^{\infty}{\zeta_{x,a}(N)z^N}.
    \end{equation}
\end{definition}
\begin{remark}[\textup{See \cite[(2.3)]{2022Dirichlet}}]
    The generating function $\psi_{x,a}\left(z\right)$ can be expressed in terms of ${\{x_n\}}$ and ${\{a_n\}}$ as follows:
    \begin{align*}
        \psi_{x,a}\left(z\right)=\sum_{n=1}^{\infty}{a_n\frac{z}{x_n-z}}.
    \end{align*}
\end{remark}
\begin{definition}[\textup{See \cite[p. 4]{2022Dirichlet}}]
    The modified sequence with respect to $\left\{a.\psi_{y,b}\right\}_{n\geq1}$ is defined as
    \begin{align*}
        \left(a.\psi_{y,b}\right)_n=a_n\psi_{y,b}(x_n).
    \end{align*}
\end{definition}
Hence, the corresponding Dirichlet series is
$$\zeta_{x,a.\psi_{y,b}}\left(N\right)=\sum_{n=1}^{\infty}\frac{a_n\psi_{y,b}(x_n)}{{x_n}^N}.$$
\begin{definition}[\textup{See \cite[p. 4]{2022Dirichlet}}]
    The convolution of two \textup{Dirichlet} series is defined as
    \begin{align*}
        \left(\zeta_{y,b}\ast\zeta_{x,a}\right)\left(N+1\right)
        =\sum_{k=1}^{N}{\zeta_{y,b}\left(k\right)\zeta_{x,a}(N+1-k)},
    \end{align*}
    and for $n$ \textup{Dirichlet} series
    $\zeta_{x^{\left(1\right)},a^{\left(1\right)}}\left(k_1\right)\ldots\zeta_{x^{\left(n\right)},a^{\left(n\right)}}$,
    the $n$-fold convolution is
    \begin{align*}
        \left(\mathop{*}\limits_{i=1}^{n} \zeta_{x^{(i)},a^{(i)}} \right)\left(N+1\right)
        =\sum{\zeta_{x^{(1)},a^{(1)}}(k_1)\cdots\zeta_{x^{(n)},a^{(n)}}(k_n)},
    \end{align*}
    where the summation is taken over the index set
    $$\left\{\left(k_1,k_2,\ldots,k_n\right):1\le k_i\le N,~\sum_{i=1}^{n}k_i=N+1\right\}.$$
\end{definition}
\begin{lemma}[\textup{See \cite[Theorem 2.2]{2022Dirichlet}}]\label{Dirichlet_jishu}
    For a collection of $n\geq2$ \textup{Dirichlet} series 
    $\left\{\zeta_{x^{\left(i\right)},a^{\left(i\right)}}\right\}_{1\le i\le n}$, 
    evaluated at the argument $N+1$, we have
    \begin{align*}
        \mathop{*}\limits_{i=1}^{n} \zeta_{x^{(i)},a^{(i)}}
        =\sum_{i=1}^{n}\zeta_{x^{\left(i\right)},a^{\left(i\right)}.\prod_{1\le k\neq i\le n}{\psi_x(k)}}.
    \end{align*}
    In particular, for $n=2$, we have
    \begin{align*}
        \zeta_{y,b}\ast\zeta_{x,a}=\zeta_{x,a.\psi_y}+\zeta_{y,b.\psi_x},
    \end{align*}
    which can be written more explicitly as:
    \begin{align*}
        \left(\zeta_{y,b}\ast\zeta_{x,a}\right)\left(N+1\right)
        =\sum_{n=1}^{\infty}\left\{\frac{a_n\psi_y(x_n)}{{x_n}^{N+1}}
        +\frac{b_n\psi_x(y_n)}{{y_n}^{N+1}}\right\}.
    \end{align*}
\end{lemma}

\subsection{Definition and properties of the Dirichlet lambda function}
\begin{definition}
    The \textup{Dirichlet lambda} function is defined as
\begin{align}
    \lambda(z) = \sum_{n=0}^\infty \frac{1}{(2n+1)^z}, \quad \operatorname{Re}(z) > 1.
\end{align}
It can be  expressed in terms of $\zeta(z)$:
\begin{align}
    \lambda(z) = \sum_{n=0}^\infty \frac{1}{(2n+1)^z} = \sum_{n=1}^\infty \frac{1}{(2n-1)^z} = \big(1 - 2^{-z}\big)\zeta(z)
\end{align}
(see \cite[p. 807]{AbramowitzStegun1964}).
\end{definition}

Furthermore, Bernoulli numbers $B_{n}$, special values of Euler polynomials at $0$, $E_{n}(0)$, and Genocchi numbers $G_{n}$ can all be used to express the values of $\lambda(z)$ at positive odd integers: For $r\in\mathbb N$
\begin{equation}
    \begin{aligned}\label{equal}
\lambda(2r) &= (-1)^{r}\frac{\pi^{2r}}{4(2r-1)!} E_{2r-1}(0) \\
&= (-1)^{r}\frac{\pi^{2r}}{4(2r)!} G_{2r} \\
&= (-1)^{r+1}\frac{(2^{2r}-1)\pi^{2r}}{2(2r)!}B_{2r}.
\end{aligned}
\end{equation}

\subsection{Lemmas}
We now introduce some lemmas that will be frequently used in the subsequent proofs.

The Bernoulli numbers  $B_{n}$ were first introduced by Jacob Bernoulli while studying formulas for sums of integer powers, they are defined by the generating function
    \begin{equation}
        \frac{t}{e^t-1}=\sum_{n=0}^{\infty}{B_n\frac{t^n}{n!}}.
    \end{equation}
\begin{lemma}[\textup{See \cite[p. 42]{AbramowitzStegun1964}}]
    Power series expansions involving the \textup{Bernoulli} numbers $B_n$ are as follows:
\\{\rm (1)}~ Hyperbolic tangent function:
\begin{align}\label{tanh}
    \tanh(x) = \sum_{k=1}^\infty 2^{2k} (2^{2k} - 1) \frac{B_{2k}}{(2k)!} x^{2k-1}, \quad |x| < \frac{\pi}{2}.
\end{align}
\\{\rm (2)}~ Tangent function:
\begin{align}\label{tan}
    \tan(x) = \sum_{k=1}^\infty 2^{2k} (2^{2k} - 1) \frac{|B_{2k}|}{(2k)!} x^{2k-1}, \quad |x| < \frac{\pi}{2}.
\end{align}
\end{lemma}

\begin{lemma}[\textup{See \cite[p. 78]{titchmarsh1986theory}}]\label{Stirling}%[Stirling's asymptotic formula for the $\Gamma$ function]
    Let $s \in \mathbb{C}$ with $s = \sigma + it$. In any vertical strip $\alpha \le \sigma \le \beta$, as $|t| \to \infty$, we have
    \begin{equation}
    \Gamma(\sigma + it) = t^{\sigma + it - \frac{1}{2}} e^{-\frac{\pi t}{2} - 
    it + \frac{i}{2}\ln\left(\sigma - \frac{1}{2}\right)} 
    (2\pi)^{\frac{1}{2}} \left( 1 + O\left( \frac{1}{t} \right) \right),
\end{equation}
\begin{equation}\label{gamma}
    |\Gamma(\sigma + it)| = (2\pi)^{\frac{1}{2}}|t|^{\sigma - \frac{1}{2}} e^{-\frac{\pi}{2}|t|} \left( 1 + O\left( \frac{1}{|t|} \right) \right).
\end{equation}
\end{lemma}
Since the reflection formula for $\Gamma(x)$ is $\Gamma(1-x) \Gamma(x)=\frac{\pi}{\sin(\pi x)}~(x\notin\mathbb{Z})$
as $\operatorname{Re}(s) \to \infty$, using  \eqref{gamma}, we obtain the following inequality:
\begin{equation}\label{sin_xiao}
    \frac{1}{\left| \sin\left( \frac{\pi s}{2k} \right) \right|} 
    = 2 \exp\left( - \frac{\pi}{2} \left| \frac{\operatorname{Re}(s)}{k} \right| \right)
     \left(1 + O\left( \frac{1}{|\operatorname{Re}(s)|} \right) \right).
\end{equation}
Similarly, another form of the reflection formula is $\Gamma\left(\frac{1}{2}+x\right)\Gamma\left(\frac{1}{2}-x\right)=\frac{\pi}{\cos(\pi x)}~(x\notin\mathbb{Z}-\frac{1}{2})$
as $\operatorname{Re}(s) \to \infty$, using  \eqref{gamma}, we obtain:
\begin{equation}\label{cos_xiao}
    \frac{1}{\left| \cos\left( \frac{\pi(s + k - 1)}{2k} \right) \right|} 
    = 2 \exp\left( - \frac{\pi}{2} \left| \frac{\operatorname{Re}(s)}{k} \right| \right) 
    \left(1 + O\left( \frac{1}{|\operatorname{Re}(s)|} \right) \right).
\end{equation}

\begin{lemma}[\textup{See \cite[p. 95]{titchmarsh1986theory}}]\label{constants}%[Polynomial growth lemma for the Riemann zeta function]
    For $\sigma>\sigma_0$, there exists a constant $C\left(\sigma\right)$ such that 
    \begin{align*}
        \left|\zeta\left(\sigma+it\right)\right|\ll\left|T\right|^{C\left(\sigma\right)}
        \quad(\left|T\right|\rightarrow\infty).
    \end{align*}
    For $\sigma>\sigma_0,$ suppose that $\left|\zeta\left(\sigma+it\right)\right|\ll\left|T\right|^{C\left(\sigma\right)}$ as $T\to\infty.$
    Since $\zeta_E(s)=(1-2^{1-s})\zeta(s)$ and $|1-2^{1-\sigma-it}|\leq1+2^{1-\sigma},$ we also have
    \begin{equation}
        \left|\zeta_E\left(\sigma+it\right)\right|\ll\left|T\right|^{C\left(\sigma\right)}\quad(\left|T\right|\rightarrow\infty).
    \end{equation}
\end{lemma}
\section{Ramanujan-type Identities for $\zeta_{E}(z,x)$}
In this section, we prove Theorems \ref{res1} and \ref{res2}, the Ramanujan-type identities for
$\zeta_{E}(z,x)$.

\subsection*{Proofs of Theorem \ref{res1}}
Let
$$a_n=b_n=\left(-1\right)^{n-1},~
x_n=\frac{{(n-1+x)}^2}{\beta},~
y_n=-\frac{{(n-1+y)}^2}{\alpha}.$$
Substituting into \eqref{dirichletzeta}, the corresponding Dirichlet series can be written in terms of the alternating Hurwitz zeta function as follows:
\begin{align*}
    \zeta_{x,a}\left(k\right)&=\sum_{n=1}^{\infty}
    {\left(-1\right)^{n-1}\frac{1}{\left(\frac{{(n-1+x)}^2}{\beta}\right)^k}}
    \\&=\beta^k\sum_{n=1}^{\infty}\frac{\left(-1\right)^{n-1}}{{(n-1+x)}^{2k}}
    \\&=\beta^k\sum_{n=0}^{\infty}\frac{\left(-1\right)^n}{{(n+x)}^{2k}}
    \\&=\beta^k\zeta_E\left(2k,x\right),
\end{align*}
\begin{align*}
    \zeta_{y,b}\left(N+1-k\right)&=\sum_{n=1}^{\infty}
    {\left(-1\right)^{n-1}\frac{1}{\left(-\frac{{(n-1+y)}^2}{\alpha}\right)^{N+1-k}}}
    \\&={(-\alpha)}^{N+1-k}\sum_{n=1}^{\infty}
    \frac{\left(-1\right)^{n-1}}{{(n-1+y)}^{2\left(N+1-k\right)}}
    \\&={(-\alpha)}^{N+1-k}\sum_{n=0}^{\infty}\frac{\left(-1\right)^n}
    {{(n+y)}^{2\left(N+1-k\right)}}
    \\&={(-\alpha)}^{N+1-k}\zeta_E\left(2\left(N+1-k\right),y\right).
\end{align*}
Furthermore, substituting ${a_n}$, ${b_n}$, ${x_n}$, ${y_n}$ into \eqref{zetasc} yields the corresponding zeta generating functions:
\begin{align*}
\psi_{x,a}\left(z\right)&=\sum_{n=1}^{\infty}{\left(-1\right)^{n-1}
\frac{z}{\frac{{(n-1+x)}^2}{\beta}-z}}
\\&=\sum_{n=0}^{\infty}\frac{\left(-1\right)^n\beta z}{{(n+x)}^2-\beta z},
\end{align*}
\begin{align*}
\psi_{y,b}\left(z\right)&=\sum_{n=1}^{\infty}{\left(-1\right)^{n-1}
\frac{z}{-\frac{{(n-1+y)}^2}{\alpha}-z}}\\&=-\sum_{n=0}^{\infty}
\frac{\left(-1\right)^n\alpha z}{{(n+y)}^2+\alpha z}.
\end{align*}
From  \eqref{digamma1_0}, we obtain
\begin{align*}
    \frac{\sqrt z}{2}\left[\widetilde{\psi}\left(x+\sqrt z\right)
    -\widetilde{\psi}\left(x-\sqrt z\right)\right]
    =\sum_{n=0}^{\infty}\frac{\left(-1\right)^nz}{\left(n+x\right)^2-z},
\end{align*}
After rearranging, substituting ${x_n}$, ${y_n}$ gives
\begin{align*}
    \psi_{x,a}(y_n)&=\frac{\sqrt{\beta y_n}}{2}\left(\widetilde{\psi}
    \left(x+\sqrt{\beta y_n}\right)-\widetilde{\psi}\left(x-\sqrt{\beta y_n}\right)\right)
    \\&=\frac{1}{2}i\sqrt{\frac{\alpha}{\beta}}\left(n-1+y\right)
    \left(\widetilde{\psi}\left(x+i\sqrt{\frac{\alpha}{\beta}}
    \left(n-1+y\right)\right)-\widetilde{\psi}\left(x-i\sqrt{\frac{\alpha}{\beta}}
    \left(n-1+y\right)\right)\right),
\end{align*}
\begin{align*}
    \psi_{y,b}(x_n)&=\frac{i\sqrt{\alpha x_n}}{2}
    \left(\widetilde{\psi}\left(y+i\sqrt{\alpha x_n}\right)
    -\widetilde{\psi}\left(y-i\sqrt{\alpha x_n}\right)\right)
    \\&=\frac{1}{2}i\sqrt{\frac{\beta}{\alpha}}\left(n-1+x\right)
    \left(\widetilde{\psi}\left(y+i\sqrt{\frac{\beta}{\alpha}}
    \left(n-1+x\right)\right)-\widetilde{\psi}
    \left(y-i\sqrt{\frac{\beta}{\alpha}}\left(n-1+x\right)\right)\right).
\end{align*}
Therefore, applying Lemma \ref{Dirichlet_jishu}, we have
\begin{align*}
    &\sum_{k=1}^{N}{\left(-\alpha\right)^{N+1-k}\beta^k
    \zeta_E\left(2k,x\right)\zeta_E\left(2\left(N+1-k\right),y\right)}
    \\&=\beta^{N+1}\frac{1}{2}i\sqrt{\frac{\alpha}{\beta}}\sum_{n=0}^{\infty}
    \frac{\left(-1\right)^n}{\left(n+x\right)^{2N+1}}\left[\widetilde{\psi}
    \left(y+i\sqrt{\frac{\alpha}{\beta}}\left(n+x\right)\right)-
    \widetilde{\psi}\left(y-i\sqrt{\frac{\alpha}{\beta}}\left(n+x\right)\right)\right]
    \\&+\left(-\alpha\right)^{N+1}\frac{1}{2}i\sqrt{\frac{\beta}{\alpha}}
    \sum_{n=0}^{\infty}\frac{\left(-1\right)^n}{\left(n+y\right)^{2N+1}}
    \left[\widetilde{\psi}\left(x+i\sqrt{\frac{\beta}{\alpha}}\left(n+y\right)\right)
    -\widetilde{\psi}\left(x-i\sqrt{\frac{\beta}{\alpha}}\left(n+y\right)\right)\right].
\end{align*}
\qed

\subsection*{\bf Proof of Theorem \ref{res2}}
Let
$$a_n=b_n=\frac{\left(-1\right)^n}{n},~x_n=-\frac{n^2}{\beta},~
y_n=\frac{n^2}{\alpha}.$$
Substituting into \eqref{dirichletzeta}, the corresponding Dirichlet series can be written in terms of $\zeta_E\left(2k+1,1\right)$ as follows:
\begin{align*}
    \zeta_{x,a}\left(k\right)&=\sum_{n=1}^{\infty}
    {\frac{\left(-1\right)^n}{n}\frac{1}{\left(-\frac{n^2}{\beta}\right)^k}}
    \\&=\left(-\beta\right)^k\sum_{n=1}^{\infty}
    \frac{\left(-1\right)^n}{n^{2k+1}}
    \\&=\left(-\beta\right)^k\sum_{n=0}^{\infty}
    \frac{\left(-1\right)^{n+1}}{{(n+1)}^{2k+1}}
    \\&=-\left(-\beta\right)^k\zeta_E\left(2k+1\right),
\end{align*}
\begin{align*}
    \zeta_{y,b}\left(N+1-k\right)&=\sum_{n=1}^{\infty}
    {\frac{\left(-1\right)^n}{n}\frac{1}{\left(\frac{n^2}{\alpha}\right)^{N+1-k}}}
    \\&=\alpha^{N+1-k}\sum_{n=1}^{\infty}
    \frac{\left(-1\right)^n}{n^{2\left(N+1-k\right)+1}}
    \\&=\alpha^{N+1-k}\sum_{n=0}^{\infty}\frac{\left(-1\right)^{n+1}}
    {{(n+1)}^{2\left(N+1-k\right)+1}}
    \\&=-\alpha^{N+1-k}\zeta_E\left(2\left(N+1-k\right)+1\right).
\end{align*}
Furthermore, substituting ${a_n}$, ${b_n}$, ${x_n}$, ${y_n}$ into \eqref{zetasc} yields the corresponding zeta generating functions:
\begin{align*}
\psi_{x,a}\left(z\right)&=\sum_{n=1}^{\infty}
{\frac{\left(-1\right)^n}{n}\frac{z}{-\frac{n^2}{\beta}-z}}
\\&=-\sum_{n=1}^{\infty}\frac{\left(-1\right)^n\beta z}{n(n^2+\beta z)},
\end{align*}
\begin{align*}
\psi_{y,b}\left(z\right)&=\sum_{n=1}^{\infty}
{\frac{\left(-1\right)^n}{n}\frac{z}{\frac{n^2}{\alpha}-z}}
\\&=\sum_{n=1}^{\infty}\frac{\left(-1\right)^n\alpha z}{n(n^2-\alpha z)}.
\end{align*}
From  \eqref{digamma1_0}, we obtain
\begin{align*}
    \sum_{n=1}^{\infty}\frac{\left(-1\right)^n\alpha z}{n(n^2-\alpha z)}
    =\frac{1}{2}\left(2\widetilde{\gamma_0}-\widetilde{\psi}\left(\sqrt{\alpha z}\right)
    -\widetilde{\psi}\left(-\sqrt{\alpha z}\right)\right),
\end{align*}
\begin{align*}
    \sum_{n=1}^{\infty}\frac{\left(-1\right)^n\beta z}{n(n^2+\beta z)}
    =-\frac{1}{2}\left(2\widetilde{\gamma_0}-\widetilde{\psi}
    \left(i\sqrt{\beta z}\right)-\widetilde{\psi}\left(-i\sqrt{\beta z}\right)\right).
\end{align*}
Substituting ${x_n}$, ${y_n}$ and combining with  \eqref{digamma1_0} and using $\alpha\beta=4\pi^2$ gives
\begin{align*}
    \psi_{x,a}(y_n)&=\frac{1}{2}\left(2\widetilde{\gamma_0}-\widetilde{\psi}
    \left(i\sqrt{\beta y_n}\right)-\widetilde{\psi}\left(-i\sqrt{\beta y_n}\right)\right)
    \\&=-\frac{1}{2}\left(\widetilde{\psi}\left(\frac{in\beta}{2\pi}\right)
    +\widetilde{\psi}\left(-\frac{in\beta}{2\pi}\right)-2\widetilde{\gamma_0}\right),
\end{align*}
\begin{align*}
    \psi_{y,b}(x_n)&=\frac{1}{2}\left(2\widetilde{\gamma_0}-\widetilde{\psi}
    \left(\sqrt{\alpha x_n}\right)-\widetilde{\psi}\left(-\sqrt{\alpha x_n}\right)\right)
    \\&=-\frac{1}{2}\left(\widetilde{\psi}\left(\frac{in\alpha}{2\pi}\right)
    +\widetilde{\psi}\left(-\frac{in\alpha}{2\pi}\right)-2\widetilde{\gamma_0}\right).
\end{align*}
Therefore, applying Lemma \ref{Dirichlet_jishu}, we have
\begin{align*}
    &\sum_{k=1}^{N}{\left(-\beta\right)^k\alpha^{N+1-k}
    \zeta_E\left(2k+1\right)\zeta_E\left(2N-2k+3\right)}
    \\&=-\sum_{n=1}^{\infty}\frac{\left(-1\right)^n}{n{(-\frac{n^2}{\beta})}^{N+1}}
    \left(\frac{1}{2}\left(\widetilde{\psi}\left(\frac{in\alpha}{2\pi}\right)
    +\widetilde{\psi}\left(-\frac{in\alpha}{2\pi}\right)-2\widetilde{\gamma_0}\right)\right)
    \\
    &\quad-\sum_{n=1}^{\infty}\frac{\left(-1\right)^n}{n{(\frac{n^2}{\alpha})}^{N+1}}
    \left(\frac{1}{2}\left(\widetilde{\psi}\left(\frac{in\beta}{2\pi}\right)
    +\widetilde{\psi}\left(-\frac{in\beta}{2\pi}\right)-2\widetilde{\gamma_0}\right)\right)
    \\
    &=-\frac{1}{2}\left(-\beta\right)^{N+1}\sum_{n=1}^{\infty}
    \frac{\left(-1\right)^n}{n^{2N+3}}\left(\widetilde{\psi}
    \left(\frac{in\alpha}{2\pi}\right)+\widetilde{\psi}\left(-\frac{in\alpha}{2\pi}\right)
    -2\widetilde{\gamma_0}\right)\\
    &\quad-\frac{1}{2}\alpha^{N+1}\sum_{n=1}^{\infty}
    \frac{\left(-1\right)^n}{n^{2N+3}}\left(\widetilde{\psi}\left(\frac{in\beta}{2\pi}\right)
    +\widetilde{\psi}\left(-\frac{in\beta}{2\pi}\right)-2\widetilde{\gamma_0}\right).
\end{align*}
This completes the proof.
\qed

\begin{corollary}
    Let $\alpha,\beta\in\mathbb{C}$ 
    with $\operatorname{Re}\left(\alpha\right)>0,~\operatorname{Re}\left(\beta\right)>0$ 
    and $\alpha\beta=4\pi^2$. Then for $m\in\mathbb{N}$, we have
    \begin{equation}
        \begin{aligned}
        &\frac{1}{\alpha}\left\{2\widetilde{\gamma_0}\zeta_E\left(3\right)
        +\sum_{n=1}^{\infty}\frac{\left(-1\right)^n}{n^3}
        \left(\widetilde{\psi}\left(\frac{in\beta}{2\pi}\right)
       +\widetilde{\psi}\left(-\frac{in\beta}{2\pi}\right)\right)\right\}
        \\
        &=\frac{1}{\beta}\left\{2\widetilde{\gamma_0}\zeta_E\left(3\right)
        +\sum_{n=1}^{\infty}\frac{\left(-1\right)^n}{n^3}\left(\widetilde{\psi}
        \left(\frac{in\alpha}{2\pi}\right)+
        \widetilde{\psi}\left(-\frac{in\alpha}{2\pi}\right)\right)\right\}.
        \end{aligned}
    \end{equation}
\end{corollary}

\begin{remark}
    \textup{Theorems \ref{res1} and \ref{res2} are Ramanujan-type identities involving $\zeta_E\left(2k,x\right)$ and $\zeta_E\left(2k+1\right)$, both containing the digamma function $\widetilde{\psi}\left(x\right)$ associated with $\zeta_E\left(z,x\right)$.}
\end{remark}

\section{Ramanujan-type Identities for $\lambda(z)$}
In this section, we prove Theorems \ref{res3} and \ref{res4}, the infinite series form for the product of tangent and hyperbolic tangent and the Ramanujan-type identity for Dirichlet lambda function $\lambda(z)$, respectively.

\subsection*{Proof of Theorem \ref{res3}}
Let $\omega=z^2$, and consider
\begin{equation}
    \begin{aligned}
    f\left(z\right)=\frac{\pi}{4} \tan{\left(z\sqrt\alpha\right)} \tanh{\left(z\sqrt\beta\right)}.
\end{aligned}
\end{equation}
From the definitions of $\tan(x)$ and $\tanh(x)$, it is evident that $f\left(z\right)$ has two classes of simple poles: $z=\frac{\left(2m+1\right)\pi}{2\sqrt\alpha}$ 
and $z=\frac{\left(2m+1\right)\pi i}{2\sqrt\beta}$ , where $m\in\mathbb{Z}$. 
Using the residue formula for a simple pole, the residues of $f\left(z\right)$ at these two classes of poles are
\begin{equation}
    \begin{aligned}\label{liushu1}
    \operatorname{Res}\left(f\left(z\right), \frac{\left(2m+1\right)\pi}{2\sqrt\alpha}\right)&
    =\lim_{z\rightarrow\frac{\left(2m+1\right)\pi}{2\sqrt\alpha}}
    {\left(z-\frac{\left(2m+1\right)\pi}{2\sqrt\alpha}\right)}
    \frac{\pi}{4} \tan{\left(z\sqrt\alpha\right)}
     \tanh{\left(z\sqrt\beta\right)}\\&=\frac{\pi}{4\sqrt\alpha}
     \tanh{\left(\frac{\left(2m+1\right)\pi}{2\sqrt\alpha}\sqrt\beta\right)}
    \\&=\frac{\pi}{4\sqrt\alpha} \tanh{\left(\left(2m+1\right)\beta\right)},
\end{aligned}
\end{equation}
\begin{equation}
\begin{aligned}\label{liushu2}
    \operatorname{Res}\left(f\left(z\right), \frac{\left(2m+1\right)\pi i}{2\sqrt\beta}\right)&
    =\lim_{z\rightarrow\frac{\left(2m+1\right)\pi i}{2\sqrt\beta}}
    {\left(z-\frac{\left(2m+1\right)\pi i}{2\sqrt\beta}\right)}
    \frac{\pi}{4} \tan{\left(z\sqrt\alpha\right)} \tanh{\left(z\sqrt\beta\right)}
    \\&=\frac{\pi}{4\sqrt\beta} \tan{\left(\frac{\left(2m+1\right)\pi i}{2\sqrt\beta}\sqrt\alpha\right)}
    \\&=-\frac{\pi}{4\sqrt\beta} \tanh{\left(\left(2m+1\right)\alpha\right)},
\end{aligned}
\end{equation}
where $\tan{z}=-i \tanh{\left(iz\right)}$.

For the first class of poles $z=\frac{\left(2m+1\right)\pi}{2\sqrt\alpha}$, where $m\in\mathbb{Z}$, according to  \eqref{liushu1}, the partial fraction decomposition of $f\left(z\right)$ corresponding to these poles is
\begin{equation}
\begin{aligned}\label{fenjie1}
    &\frac{\pi}{4\sqrt\alpha}\sum_{m=0}^{\infty}\left(\frac{\tanh{\left(\left(2m+1\right)\beta\right)}}
    {z-\frac{\left(2m+1\right)\pi}{2\sqrt\alpha}}
    +\frac{\tanh{\left(\left(2m+1\right)\beta\right)}}
   {z+\frac{\left(2m+1\right)\pi}{2\sqrt\alpha}}\right)
    \\&=\sum_{m=0}^{\infty}\frac{\left(2m+1\right)\beta \tanh{\left(\left(2m+1\right)\beta\right)}}
   {z^2-\left(2m+1\right)^2\beta}.
\end{aligned}
\end{equation}
For the second class of poles $z=\frac{\left(2m+1\right)\pi}{2\sqrt\beta}$, where $m\in\mathbb{Z}$, according to  \eqref{liushu2}, the partial fraction decomposition of $f\left(z\right)$ corresponding to these poles is
\begin{equation}
\begin{aligned}\label{fenjie2}
    &-\frac{\pi}{4\sqrt\beta}\sum_{m=0}^{\infty}\left(\frac{\tanh{\left(\left(2m+1\right)\alpha\right)}}
    {z-\frac{\left(2m+1\right)\pi i}{2\sqrt\beta}}
    +\frac{\tanh{\left(\left(2m+1\right)\alpha\right)}}
    {z+\frac{\left(2m+1\right)\pi i}{2\sqrt\beta}}\right)
    \\&=-\sum_{m=0}^{\infty}\frac{\left(2m+1\right)\alpha \tanh{\left(\left(2m+1\right)\alpha\right)}}
    {z^2+\left(2m+1\right)^2\alpha}.
\end{aligned}
\end{equation}
Applying Mittag-Leffler theorem (see \cite[p. 205]{Conway}), \eqref{fenjie1} and \eqref{fenjie2}, there exists an entire function $g\left(z\right)$ satisfying:
\begin{equation}
\begin{aligned}
    &\frac{\pi}{4} \tan{\left(z\sqrt\alpha\right)} \tanh{\left(z\sqrt\beta\right)}
    =\\&\sum_{m=0}^{\infty}\left\{\frac{\left(2m+1\right)
    \beta \tanh{\left(\left(2m+1\right)\beta\right)}}{z^2-\left(2m+1\right)^2\beta}
    -\frac{\left(2m+1\right)\alpha \tanh{\left(\left(2m+1\right)\alpha\right)}}
    {z^2+\left(2m+1\right)^2\alpha}\right\}+g\left(z\right).
\end{aligned}
\end{equation}
Since $f(z)$ and the partial fraction series are bounded in the complex plane (except at the poles), $g(z)$ is a bounded entire function. 
By Liouville's Theorem, a bounded entire function must be constant, i.e., $g(z) \equiv C$ ($C$ constant).
We now show that the constant $C = 0$: Evaluating at the special point $z = 0$, the left-hand side $f(0) = \frac{\pi}{4} \tan(0) \tanh(0) = 0$,
 and the right-hand side partial fraction series at $z = 0$ vanishes because the terms cancel pairwise due to the modular symmetry $\alpha\beta = \frac{\pi^2}{4}$. Hence $0 = 0 + C$, so $C = 0$, $g(z) \equiv 0$.
 Thus, Theorem \ref{res3} is proved.
\qed

\begin{remark}
    \textup{Theorem \ref{res3} establishes, under the modular symmetry $\alpha\beta=\frac{\pi^2}{4}$, an infinite series expansion for the product of tangent and hyperbolic tangent functions, converting it into a computable rational series.}
\end{remark}

\subsection*{\bf Proof of Theorem \ref{res4}}
Expanding both sides of  \eqref{eq8} in a Taylor series around $z=0$, using  \eqref{tanh} and \eqref{tan}, we have
\begin{equation}\label{eq4_1}
    \begin{aligned}
        &\frac{\pi}{4}\sum_{j=1}^{\infty}{\frac{2^{2j}\left(2^{2j}-1\right)}
        {\left(2j\right)!}\left|B_{2j}\right|{(\sqrt{\omega\alpha})}^{2j-1}}
        \cdot\sum_{k=1}^{\infty}{\frac{2^{2k}\left(2^{2k}-1\right)}
        {\left(2k\right)!}B_{2k}\left(\sqrt{\omega\beta}\right)^{2k-1}}
        \\&=-\sum_{m=0}^{\infty}\left\{\frac{\tanh{\left(\left(2m+1\right)\beta\right)}}
        {2m+1}\sum_{r=0}^{\infty}\left(\frac{\omega}{\left(2m+1\right)^2\beta}\right)^r
        +\frac{\tanh{\left(\left(2m+1\right)\alpha\right)}}{2m+1}\sum_{r=0}^{\infty}
        \left(-\frac{\omega}{\left(2m+1\right)^2\alpha}\right)^r\right\}.
    \end{aligned}
\end{equation}
In  \eqref{eq4_1}, equating the coefficients of $\omega^{r}$ for $r\in\mathbb N$ on both sides yields
\begin{equation}\label{eq4_2}
    \begin{aligned} 
        &-\frac{1}{2}\sum_{k=0}^{r+1}{\frac{2^{2k}\left(2^{2k}-1\right)}{\left(2k\right)!}
        \frac{2^{2\left(r+1-k\right)}\left(2^{2\left(r+1-k\right)}-1\right)}
        {\left(2\left(r+1-k\right)\right)!}\left|B_{2k}\right|B_{2\left(r+1-k\right)}\alpha^k}
        \beta^{r+1-k}\\
        &=-\beta^{-r}\sum_{m=0}^{\infty}\frac{\tanh{\left(\left(2m+1\right)\beta\right)}}
        {\left(2m+1\right)^{2r+1}}+(-1)^r\alpha^{-r}\sum_{m=0}^{\infty}
        \frac{\tanh{\left(\left(2m+1\right)\alpha\right)}}{\left(2m+1\right)^{2r+1}}
        \\&=-\beta^{-r}\sum_{m=0}^{\infty}\frac{1}{\left(2m+1\right)^{2r+1}}
        \left(1-\frac{2}{e^{2\left(2m+1\right)\beta}+1}\right)
        \\&\quad+(-1)^r\alpha^{-r}
        \sum_{m=0}^{\infty}\frac{1}{\left(2m+1\right)^{2r+1}}
        \left(1-\frac{2}{e^{2\left(2m+1\right)\alpha}+1}\right)
        \\&=-2\beta^{-r}\left(\frac{1}{2}\lambda\left(2r+1\right)-\sum_{m=0}^{\infty}
        \frac{\left(2m+1\right)^{-2r-1}}{e^{2\left(2m+1\right)\beta}+1}\right)
        \\
        &\quad+2(-1)^r\alpha^{-r}\left(\frac{1}{2}\lambda\left(2r+1\right)-
        \sum_{m=0}^{\infty}\frac{\left(2m+1\right)^{-2r-1}}{e^{2\left(2m+1\right)\alpha}+1}\right).
    \end{aligned}
\end{equation}
According to formula \eqref{equal}, the coefficient of $\omega^{r}$ for $r\in\mathbb N$ on the left-hand side of  \eqref{eq4_1} has the following equivalent expressions:
\begin{equation}\label{4.9}
    \begin{aligned}
        &-\frac{1}{2}\sum_{k=1}^{r}{\frac{2^{2r+2}\left(2^{2k}-1\right)
        \left(2^{2\left(r+1-k\right)}-1\right)\left|B_{2k}\right|B_{2\left(r+1-k\right)}}
        {\left(2k\right)!\left(2\left(r+1-k\right)\right)!}\alpha^k\beta^{r+1-k}}
        \\&=-2^{2r+1}\sum_{k=1}^{r}{\frac{k(r+1-k)E_{2k-1}(0)E_{2(r+1-k)-1}(0)}
        {\left(2k\right)!\left(2\left(r+1-k\right)\right)!}\alpha^k\beta^{r+1-k}}
        \\&=-2^{2r-1}\sum_{k=1}^{r}{\frac{G_{2k}G_{2\left(r+1-k\right)}}
        {\left(2k\right)!\left(2\left(r+1-k\right)\right)!}\alpha^k\beta^{r+1-k}}.
    \end{aligned}
\end{equation}
Note that
 \begin{equation}\label{4.9+}
  \begin{aligned}
 	&-\frac{1}{2}\sum_{k=0}^{r+1}{\frac{2^{2k}\left(2^{2k}-1\right)}{\left(2k\right)!}
 		\frac{2^{2\left(r+1-k\right)}\left(2^{2\left(r+1-k\right)}-1\right)}
 		{\left(2\left(r+1-k\right)\right)!}\left|B_{2k}\right|B_{2\left(r+1-k\right)}\alpha^k}
 	\beta^{r+1-k} \\
	&=-\frac{1}{2}\sum_{k=1}^{r}{\frac{2^{2r+2}\left(2^{2k}-1\right)
			\left(2^{2\left(r+1-k\right)}-1\right)\left|B_{2k}\right|B_{2\left(r+1-k\right)}}
		{\left(2k\right)!\left(2\left(r+1-k\right)\right)!}\alpha^k\beta^{r+1-k}}.
\end{aligned}
\end{equation}
Finally, by combing (\ref{eq4_2}), (\ref{4.9}) and(\ref{4.9+}), we obtain (\ref{Th1}).
\qed

\begin{remark}
    \textup{Theorem \ref{res4} is a generalization of Ramanujan identities under the modular symmetry $\alpha\beta = \frac{\pi^2}{4}$. The proof employs the method of matching Taylor series coefficients, relating the Dirichlet lambda function, Bernoulli numbers, special values of Euler polynomials at $0$, and Genocchi numbers as equivalent finite sums.}
\end{remark}

\section{Convolutions of the alternating Hurwitz zeta functions}
In this section, under the modular symmetry $\alpha\beta=\pi^2$, we first confirm  the definitions of the alternating even-order and odd-order Hurwitz kernels (see Definitions \ref{alternating_hurwitz_kernel_even} and \ref{alternating_hurwitz_kernel_odd}). Then by using Cauchy's residue theorem and the theory of Dirichlet series, a series of convolution identities involving the alternating Hurwitz zeta function are derived (see Theorems  \ref{res5}--\ref{res9}).

\subsection*{Proof of Definition \ref{alternating_hurwitz_kernel_even}}
Let $x\in\mathbb{R}^+$, $a\in\mathbb{C}$ and $k\in\mathbb{N}$.
Define the integral representation of the alternating even-order Hurwitz kernel as follows:
        \begin{align*}
            F(x,a;k)=\frac{1}{2i\pi}\int_{\left(c\right)}
            {\frac{\zeta_E\left(1-s,a\right)}
            {2k \cos{\left(\frac{\pi\left(s+k-1\right)}{2k}\right)}}x^{-s}{\rm d}s},
        \end{align*}
where the integration path $(c)$ is the vertical line $\mathrm{Re}(s)=c~(1<c<2)$ in the complex plane.
We construct a rectangular contour and evaluate the integral by shifting the line of integration. In the complex plane, consider the rectangular region bounded by the four line segments $[c-iT,c+iT]$, $[c+iT,c'+iT]$, $[c'+iT,c'-iT]$, $[c'-iT,c-iT]$, where $c' = 1 - c$. 
\begin{figure}[h]
\centering
\begin{tikzpicture}[scale=1.1]

% Axes
\draw[->, thick, black] (-4.5, 0) -- (4.5, 0) node[right] {$\operatorname{Re}(s)$};
\draw[->, thick, black] (0, -3.5) -- (0, 3.5) node[above] {$\operatorname{Im}(s)$};

% Parameters
\def\c{1.2}       
\def\cp{1 - \c}   
\def\T{2.8}       

% Vertical dashed lines
\draw[dashed, thick, black] (\c, -3.5) -- (\c, 3.5);
\node[right=5pt, yshift=-8pt] at (\c, 0) {$\operatorname{Re}(s)=c$};
\fill (\c, 0) circle (2pt) node[above right=3pt] {$(c, 0)$};

\draw[dashed, thick, black] (\cp, -3.5) -- (\cp, 3.5);
\node[left=5pt, yshift=-8pt] at (\cp, 0) {$\operatorname{Re}(s)=c'$};
\fill (\cp, 0) circle (2pt) node[above left=3pt] {$(c', 0)$};

% Vertices
\fill (\c, \T)  circle (2pt) node[above right=2pt] {$(c,iT)$};
\fill (\cp, \T) circle (2pt) node[above left=2pt]  {$(c',iT)$};
\fill (\cp,-\T) circle (2pt) node[below left=2pt]  {$(c',-iT)$};
\fill (\c,-\T)  circle (2pt) node[below right=2pt] {$(c,-iT)$};

% Contour
\draw[dashed, thick, red!80!black] (\c, \T) -- (\cp, \T) -- (\cp,-\T) -- (\c,-\T) -- cycle;

\end{tikzpicture}
\caption{Diagram of the contour integration}
\end{figure}
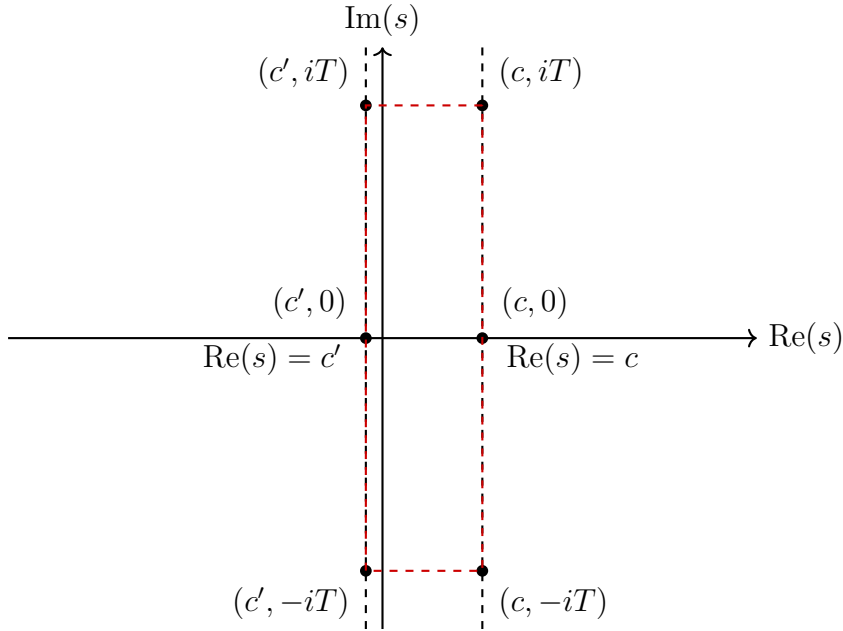

Inside this contour, the integrand has a simple pole at $s=1$. Its residue is calculated as
\begin{align*}
    \operatorname{Res}\left(\frac{\zeta_E\left(1-s,a\right)}{2k \cos{\left(\frac{\pi\left(s+k-1\right)}{2k}\right)}}x^{-s}\right)_{s=1}
    &=\lim_{s\to 1} \frac{(s-1) \zeta_E(1-s,a) x^{-s}}{2k \cos{\left(\frac{\pi\left(s+k-1\right)}{2k}\right)}}
\\&=\frac{1}{2} \lim_{s\to 1} \frac{(s-1)}{2k \cos{\left(\frac{\pi\left(s+k-1\right)}{2k}\right)}} x^{-1}
\\&=\frac{1}{2} \lim_{s\to 1} \frac{1}{2k \left(-\frac{\pi}{2k}\right) 
	\sin{\left(\frac{\pi (s+k-1)}{2k} \right)} }x^{-1} 
\\
&=\frac12\left( \frac{1}{-\pi\sin\left( \frac{\pi k}{2k} \right)} \right) x^{-1} \\
&= -\frac{1}{2 \pi x} .
\end{align*}
By Cauchy's residue theorem, the contour integral equals the sum of residues inside the contour, i.e.,
\begin{align*}
    &\frac{1}{2i\pi}\left[\int_{c-iT}^{c+iT} + \int_{c+iT}^{c'+iT} + \int_{c'+iT}^{c'-iT} + \int_{c'-iT}^{c-iT}\right]
\frac{\zeta_E(1-s,a)}{2k \cos{\left(\frac{\pi\left(s+k-1\right)}{2k}\right)}}x^{-s} \mathrm{d}s.
\\&= -\frac{1}{2 \pi x} .
\end{align*}
Using Lemma \ref{constants} and  \eqref{cos_xiao}, as $T\rightarrow\infty$, the integrals along the horizontal segments tend to $0$. Hence,
\begin{align}
    &\frac{1}{2i\pi}\int_{\left(c\right)}\frac{\zeta_E(1-s,a)}{2k \cos{\left(\frac{\pi\left(s+k-1\right)}{2k}\right)}}x^{-s} \mathrm{d}s
        \\&=\frac{1}{2i\pi}\int_{\left(c'\right)}\frac{\zeta_E(1-s,a)}{2k \cos{\left(\frac{\pi\left(s+k-1\right)}{2k}\right)}}x^{-s}{\rm d}s-\frac{1}{2 \pi x}.
\end{align}
Since the alternating Hurwitz zeta function $\zeta_E(z,x)$ converges absolutely for $\operatorname{Re}(z)>0$, we may interchange the infinite sum and the complex integral, obtaining
\begin{align}\label{you1}
    \frac{1}{2i\pi}\int_{\left(c'\right)}\frac{\zeta_E(1-s,a)}{2k \cos{\left(\frac{\pi\left(s+k-1\right)}{2k}\right)}}x^{-s} \mathrm{d}s=
    \sum_{n=0}^{\infty} \int_{(c')} \frac{(-1)^n(n+a)^{s-1} x^{-s}}{2k \cos{\left(\frac{\pi\left(s+k-1\right)}{2k}\right)}} \mathrm{d}s.
\end{align}
For the integral on the right-hand side of  \eqref{you1}, we compute the residue at $s=-2kp+1$ for any positive integer $p$:
\begin{align*}
    &\operatorname{Res}\left( \frac{(-1)^n (n+a)^{s-1} x^{-s}}{2k \cos{\left(\frac{\pi\left(s+k-1\right)}{2k}\right)}} \right)_{s=-2kp+1}
    \\&=\lim_{s\rightarrow -2kp+1}\frac{(s+2kp-1)(-1)^n (n+a)^{s-1} x^{-s}}{2k \cos{\left(\frac{\pi\left(s+k-1\right)}{2k}\right)}}
    \\&=\frac{(-1)^n (n+a)^{-2kp} x^{2kp-1}}{(-1)^{p+1} \pi}
    \\&=-\frac{(-1)^n}{ \pi x}\left( -\frac{x^{2k}}{(n+a)^{2k}}\right)^p .
\end{align*}
The contour $\operatorname{Re}(s) = c'$ ($-1 < c' < 0$) encloses all simple poles $s = -2kp+1$~($p \geq 1$) to its left. Shifting the contour to negative infinity and applying Cauchy's residue theorem, for $|x| < 1$ and using the geometric series sum, we have:
\begin{equation}\label{eq22}
    \begin{aligned}
&\frac{1}{2i\pi} \int_{(c')} \frac{(-1)^n (n+a)^{s-1} x^{-s}}{2k \cos{\left(\frac{\pi\left(s+k-1\right)}{2k}\right)}} {\rm d}s
\\&= \sum_{p=1}^{\infty} \operatorname{Res}\left( \frac{(-1)^n (n+a)^{s-1} x^{-s}}{2k \cos{\left(\frac{\pi\left(s+k-1\right)}{2k}\right)}} \right)_{s=-2kp+1}
\\&=-\sum_{p=1}^{\infty}\frac{(-1)^n}{ \pi x}\left( -\frac{x^{2k}}{(n+a)^{2k}}\right)^p
\\&=\frac{(-1)^n}{ \pi x}\left(\frac{{x}^{2k}}{{(n+a)}^{2k}+{x}^{2k}}\right) .
\end{aligned}
\end{equation}
Note that the integral on the right-hand side of  \eqref{you1} is analytic in the half-plane $\operatorname{Re}(1-s)>0$. Hence, we conclude that  \eqref{eq22} holds for all complex $s\in\mathbb{C}$ satisfying $\operatorname{Re}(1-s)>0$. 
Consequently, the following equality holds for all positive real numbers $x\in\mathbb{R}^+$:
\begin{align*}
    F(x,a;k)&=\frac{1}{2i\pi}\int_{\left(c\right)}
            {\frac{\zeta_E\left(1-s,a\right)}
            {2k \cos{\left(\frac{\pi\left(s+k-1\right)}{2k}\right)}}x^{-s}{\rm d}s}
            \\&=-\frac{1}{2\pi x}+\frac{x^{2k-1}}{\pi}\sum_{n=0}^{\infty}
            \frac{\left(-1\right)^n}{{(n+a)}^{2k}+x^{2k}},
\end{align*}
which is (\ref{Def1.19}).
\qed

\subsection*{Proof of Definition \ref{alternating_hurwitz_kernel_odd}}
Define the integral representation of the alternating odd-order Hurwitz kernel as follows:
        \begin{equation}
            G\left(x,a;k\right)=\frac{1}{2i\pi}\int_{\left(c\right)}
            {\frac{\zeta_E\left(1-s,a\right)}{2k \sin{\left(\frac{\pi s}{2k}\right)}}x^{-s}{\rm d}s},
        \end{equation}
where the integration path $(c)$ is the vertical line $\mathrm{Re}(s)=c~(1<c<2)$.
We construct a rectangular contour and evaluate the integral by shifting the line of integration. In the complex plane, consider the rectangular region bounded by the four line segments $[c-iT,c+iT]$, $[c+iT,c'+iT]$, $[c'+iT,c'-iT]$, $[c'-iT,c-iT]$, where $c' = 1 - c$. 
Inside this contour, the integrand has a simple pole only at $s=0$. Using the residue formula for a simple pole, we compute:
\begin{align*}
    \operatorname{Res}\left(\frac{\zeta_E\left(1-s,a\right) x^{-s}}{2k \sin\left(\frac{\pi s}{2k}\right)}\right)_{s=0}
    &=\lim_{s\to 0} \frac{(s-0)\zeta_E(1-s,a) x^{-s}}{2k \sin\left(\frac{\pi s}{2k}\right)}
\\&= \lim_{s\to 0} \frac{\zeta_E(1-s,a)}{2k  (\frac{\pi}{2k}) \cos\left(\frac{\pi s}{2k}\right)}
\\&= \frac{\zeta_E(1,a)}{\pi} \\&= -\frac{\widetilde{\psi}(a)}{\pi}.
\end{align*}
By Cauchy's residue theorem, the contour integral equals the sum of residues inside the contour, i.e.,
\begin{align*}
    &\frac{1}{2i\pi}\left[\int_{c-iT}^{c+iT} + \int_{c+iT}^{c'+iT} + \int_{c'+iT}^{c'-iT} + \int_{c'-iT}^{c-iT}\right]
\frac{\zeta_E(1-s,a)}{2k \sin\left(\frac{\pi s}{2k}\right)}x^{-s} \mathrm{d}s
\\&= -\frac{\widetilde{\psi}(a)}{\pi}.
\end{align*}
Using Lemma \ref{constants} and  \eqref{sin_xiao}, as $T\rightarrow\infty$, the integrals along the horizontal segments tend to $0$. Hence,
\begin{equation}
    \begin{aligned}
    &\frac{1}{2i\pi}\int_{\left(c\right)}\frac{\zeta_E(1-s,a)}{2k \sin\left(\frac{\pi s}{2k}\right)}x^{-s} \mathrm{d}s
        \\&=\frac{1}{2i\pi}\int_{\left(c'\right)}\frac{\zeta_E(1-s,a)}{2k \sin\left(\frac{\pi s}{2k}\right)}x^{-s}{\rm d}s-\frac{\widetilde{\psi}(a)}{\pi}.
\end{aligned}
\end{equation}
Since the alternating Hurwitz zeta function $\zeta_E(z,x)$ converges absolutely for $\operatorname{Re}(z)>0$, we may interchange the infinite sum and the complex integral, obtaining
\begin{align}\label{you2}
    \frac{1}{2i\pi}\int_{\left(c'\right)}\frac{\zeta_E(1-s,a)}{2k \sin\left(\frac{\pi s}{2k}\right)}x^{-s} \mathrm{d}s=
    \sum_{n=0}^{\infty} \int_{(c')} \frac{(-1)^n(n+a)^{s-1} x^{-s}}{2k \sin\left(\frac{\pi s}{2k}\right)} \mathrm{d}s.
\end{align}
For the integral on the right-hand side of  \eqref{you2}, we compute the residue at $s=-2kp$ for any positive integer $p$:
\begin{align*}
    &\operatorname{Res}\left( \frac{(-1)^n (n+a)^{s-1} x^{-s}}{2k \sin\left(\frac{\pi s}{2k}\right)} \right)_{s=-2kp}
    \\&=\lim_{s\rightarrow -2kp}\frac{(s+2kp)(-1)^n (n+a)^{s-1} x^{-s}}{2k \sin\left(\frac{\pi s}{2k}\right)}
    \\&=\frac{(-1)^n (n+a)^{-2kp-1} x^{2kp}}{(-1)^p \pi}
    \\&=\frac{(-1)^n}{ \pi (n+a)}\left( -\frac{x^{2k}}{(n+a)^{2k}}\right)^p .
\end{align*}
The contour $\operatorname{Re}(s) = c'$ ($-1 < c' < 0$) encloses all simple poles $s = -2kp~$($p \geq 1$) to its left. Shifting the contour to negative infinity and applying Cauchy's residue theorem, for $|x| < 1$ and using the geometric series sum, we have:
\begin{equation}
    \begin{aligned}\label{eq1_1}
&\frac{1}{2i\pi} \int_{(c')} \frac{(-1)^n (n+a)^{s-1} x^{-s}}{2k \sin\left(\frac{\pi s}{2k}\right)} {\rm d}s
\\&= \sum_{p=1}^{\infty} \operatorname{Res}\left( \frac{(-1)^n (n+a)^{s-1} x^{-s}}{2k \sin\left(\frac{\pi s}{2k}\right)} \right)_{s=-2kp}
\\&=\sum_{p=1}^{\infty}\frac{(-1)^n}{ \pi (n+a)}\left( -\frac{x^{2k}}{(n+a)^{2k}}\right)^p
\\&=\frac{(-1)^n}{ \pi (n+a)}\left(\frac{-{x}^{2k}}{{(n+a)}^{2k}+{x}^{2k}}\right).
\end{aligned}
\end{equation}

Note that the integral on the right-hand side of  \eqref{you2} is analytic in the half-plane $\operatorname{Re}(1-s)>0$. Hence, we conclude that  \eqref{eq1-1} holds for all complex $s\in\mathbb{C}$ satisfying $\operatorname{Re}(1-s)>0$. 
Consequently, the following equality holds for all positive real numbers $x\in\mathbb{R}^+$:
\begin{align*}
    G\left(x,a;k\right)&=\frac{1}{2i\pi}\int_{\left(c\right)}
            {\frac{\zeta_E\left(1-s,a\right)}{2k \sin{\left(\frac{\pi s}{2k}\right)}}x^{-s}{\rm d}s}
            \\&=\frac{-\widetilde{\psi}\left(a\right)}{\pi}-\frac{x^{2k}}
            {\pi}\sum_{n=0}^{\infty}\frac{\left(-1\right)^n}
            {\left(n+a\right)\left({(n+a)}^{2k}+x^{2k}\right)},
\end{align*}
which is (\ref{Def1.20}).
\qed

\subsection*{Proof of Theorem \ref{res5}}
We give two proofs.

(1)  Since $\zeta_{E}(z,x)$ converges absolutely for $\operatorname{Re}(z)>0$, substituting the integral definition of the alternating even-order Hurwitz kernel into the series on the left-hand side of the theorem and interchanging the infinite sum and the complex integral yields
    \begin{equation}
        \begin{aligned}\label{jifen1}
            &\sum_{n=0}^{\infty}\frac{\left(-1\right)^n
            F_\alpha\left(n+b,a;k\right)}{\left(n+b\right)^
            {2k\left(N+1\right)-1}}\\&=\sum_{n=0}^{\infty}
            \frac{\left(-1\right)^n}{\left(n+b\right)^
            {2k\left(N+1\right)-1}}\left(\frac{1}{2i\pi}
            \int_{\left(c\right)}{\frac{\zeta_E\left(1-s,a\right)}
            {2k \cos{\left(\frac{\pi\left(s+k-1\right)}{2k}\right)}}
            \left(\frac{\alpha\left(n+b\right)}{\pi}\right)^{-s}{\rm d}s}\right)
            \\&=\frac{1}{2i\pi}\int_{\left(c\right)}
            \frac{\zeta_E\left(1-s,a\right)\zeta_E\left(2kN+2k-1+s,b\right)}
            {2k \cos{\left(\frac{\pi\left(s+k-1\right)}{2k}\right)}}
            \left(\frac{\alpha}{\pi}\right)^{-s}{\rm d}s,
        \end{aligned}
\end{equation}
    where the integration path $(c)$ is the vertical line $\mathrm{Re}(s)=c, (1<c<2)$.
    We construct a rectangular contour and evaluate the integral by shifting the line of integration. In the complex plane, consider the rectangular region bounded by the four line segments $[c-iT,c+iT]$, $[c+iT,d+iT]$, $[d+iT,d-iT]$, $[d-iT,c-iT]$, where $d=-c-2kN-2k+2$. 
    Since the alternating Hurwitz zeta function is analytic in the entire complex plane, all poles of this integrand come from the denominator
    $$2k \cos{\left(\frac{\pi\left(s+k-1\right)}{2k}\right)}.$$ The poles are at integers $s=-2kp+1~(p\in\left\{0,1,\ldots,N+1\right\})$, all of which are simple poles. For a simple pole $s=-2kp+1$, using the residue formula we compute:
    \begin{align*}
        R_{-2kp+1}&=\lim_{s\rightarrow-2kp+1}{\frac{\left(s+2kp-1\right)\zeta_E\left(1-s,a\right)\zeta_E\left(2kN+2k-1+s,b\right)}{2k \cos{\left(\frac{\pi\left(s+k-1\right)}{2k}\right)}}}\left(\frac{\alpha}{\pi}\right)^{-s}
        \\&=\frac{\zeta_E\left(1-\left(-2kp+1\right),a\right)\zeta_E\left(2kN+2k-1+\left(-2kp+1\right),b\right)}{-\pi s i n{\left(\frac{\pi\left(\left(-2kp+1\right)+k-1\right)}{2k}\right)}}\left(\frac{\alpha}{\pi}\right)^{2kp-1}
        \\&=\frac{\zeta_E\left(2kp,a\right)\zeta_E\left(2k\left(N+1-p\right),b\right)}{-\pi s i n{\left(-p\pi+\frac{\pi}{2}\right)}}\left(\frac{\alpha}{\pi}\right)^{2kp-1}
        \\&=\frac{\left(-1\right)^{p+1}}{\pi}\left(\frac{\alpha}{\pi}\right)^{2kp-1}\zeta_E\left(2kp,a\right)\zeta_E\left(2k\left(N+1-p\right),b\right).
    \end{align*}
    By Cauchy's residue theorem, the contour integral equals the sum of residues inside the contour, i.e.,
    \begin{equation}
    \begin{aligned}\label{cauchy1}
        &\frac{1}{2i\pi}\left[\int_{c-iT}^{c+iT}+\int_{c+iT}^{d+iT}+\int_{d+iT}^{d-iT}+\int_{d-iT}^{c-iT}\right]
\frac{\zeta_E\left(1-s,a\right)\zeta_E\left(2kN+2k-1+s,b\right)}{2k \cos{\left(\frac{\pi\left(s+k-1\right)}{2k}\right)}}
\left(\frac{\alpha}{\pi}\right)^{-s}{\rm d} s\\&=\sum_{p=0}^{N+1}R_{-2kp+1}.
    \end{aligned}
    \end{equation}
    Using Lemma \ref{constants} and  \eqref{cos_xiao}, 
    as $T\rightarrow\infty$, the integrals along the horizontal segments tend to $0$. Hence,
    \begin{equation}
        \begin{aligned}\label{0_1}
        &\frac{1}{2i\pi}\int_{\left(c\right)}\frac{\zeta_E\left(1-s,a\right)\zeta_E\left(2kN+2k-1+s,b\right)}{2k \cos{\left(\frac{\pi\left(s+k-1\right)}{2k}\right)}}\left(\frac{\alpha}{\pi}\right)^{-s}
        {\rm d}s
        \\&=\frac{1}{2i\pi}\int_{\left(d\right)}\frac{\zeta_E\left(1-s,a\right)\zeta_E\left(2kN+2k-1+s,b\right)}{2k \cos{\left(\frac{\pi\left(s+k-1\right)}{2k}\right)}}\left(\frac{\alpha}{\pi}\right)^{-s}
        {\rm d}s+\sum_{p=0}^{N+1}R_{-2kp+1},
    \end{aligned}
    \end{equation}
    Next, making the substitution $s\rightarrow-s-2kN-2k+2$ and using the modular symmetry $\alpha\beta=\pi^2$, we obtain
    \begin{equation}
        \begin{aligned}\label{tihuan1}
        &\int_{\left(d\right)}\frac{\zeta_E\left(1-s,a\right)
        \zeta_E\left(2kN+2k-1+s,b\right)}
        {2k \cos{\left(\frac{\pi\left(s+k-1\right)}
        {2k}\right)}}\left(\frac{\alpha}{\pi}\right)^{-s}{\rm d}s
        \\&=\left(-1\right)^N\left(\frac{\alpha}
        {\pi}\right)^{2kN+2k-2}\int_{\left(c\right)}
        {\frac{\zeta_E\left(1-s,a\right)\zeta_E\left(2kN+2k-1+s,b\right)}
        {2k \cos{\left(\frac{\pi\left(s+k-1\right)}
        {2k}\right)}}\left(\frac{\beta}{\pi}\right)^{-s}}{\rm d}s,
    \end{aligned}
    \end{equation}
    Combining  \eqref{0_1} and \eqref{tihuan1}, we have
    \begin{equation}
        \begin{aligned}
        &\frac{1}{2i\pi}\int_{\left(c\right)}
        \frac{\zeta_E\left(1-s,a\right)
        \zeta_E\left(2kN+2k-1+s,b\right)}
        {2k \cos{\left(\frac{\pi\left(s+k-1\right)}
        {2k}\right)}}\left(\frac{\alpha}{\pi}\right)^{-s}{\rm d}s
        \\&=\left(-1\right)^N\left(\frac{\alpha}{\pi}\right)^
        {2kN+2k-2}\frac{1}{2\pi i}\int_{\left(c\right)}
        {\frac{\zeta_E\left(1-s,b\right)\zeta_E\left(2kN+2k-1+s,a\right)}
        {2k \cos{\left(\frac{\pi\left(s+k-1\right)}{2k}\right)}}
        \left(\frac{\beta}{\pi}\right)^{-s}}{\rm d}s
        \\
        &\quad+\sum_{p=0}^{N+1}R_{-2kp+1}.
    \end{aligned}
    \end{equation}
    Recalling Definition \ref{alternating_hurwitz_kernel_even}, we obtain
    \begin{equation}
        \begin{aligned}\label{eq1-1}
        &\sum_{n=0}^{\infty}\frac{\left(-1\right)^n
        F_\alpha\left(n+b,a;k\right)}{\left(n+b\right)
        ^{2k\left(N+1\right)-1}}
        \\&=\left(-1\right)^N\left(\frac{\alpha}{\pi}\right)
        ^{2kN+2k-2}\sum_{n=0}^{\infty}\frac{{(-1)}^n
        F_\beta(n+a,b;k)}{\left(n+a\right)^{2k(N+1)-1}}
        \\
        &\quad+\sum_{p=0}^{N+1}{\frac{\left(-1\right)^{p+1}}
        {\pi}\left(\frac{\alpha}{\pi}\right)^{2kp-1}
        \zeta_E\left(2kp,a\right)\zeta_E
        \left(2k\left(N+1-p\right),b\right)}.
    \end{aligned}
    \end{equation}

    Since $\alpha\beta=\pi^2$, multiplying both sides by $\beta^{k\left(N+1\right)-1}$ 
    and rearranging, we first obtain (\ref{gen1}), the functional form of Theorem \ref{res5}.

Then in   \eqref{eq1-1}, substituting the series expressions for $F_\alpha(n+b,a;k)$ and $F_\beta(n+a,b;k)$ (see Definition \ref{alternating_hurwitz_kernel_even}) gives
\begin{equation}
    \begin{aligned}
        &\sum_{n=0}^{\infty}\frac{\left(-1\right)^nF_\alpha\left(n+b,a;k\right)}{\left(n+b\right)^{2k\left(N+1\right)-1}}
        \\&=\sum_{n=0}^{\infty}{\frac{\left(-1\right)^n}{\left(n+b\right)^{2k\left(N+1\right)-1}}\left(-\frac{1}{2\pi\left(\frac{\alpha\left(n+b\right)}{\pi}\right)}+\frac{1}{\pi}\sum_{i=0}^{\infty}\frac{{(-1)}^i\left(\frac{\alpha\left(n+b\right)}{\pi}\right)^{2k-1}}{{(i+a)}^{2k}+\left(\frac{\alpha\left(n+b\right)}{\pi}\right)^{2k}}\right)}
        \\&=\sum_{n=0}^{\infty}{\frac{\left(-1\right)^n\alpha^{k-1}}{\left(n+b\right)^{2kN}}\sum_{i=0}^{\infty}\frac{{(-1)}^{i}}{\beta^k{(i+a)}^{2k}+\alpha^k{(n+b)}^{2k}}}-\frac{1}{2\alpha}\sum_{n=0}^{\infty}\frac{\left(-1\right)^n}{\left(n+b\right)^{2k\left(N+1\right)}},
    \end{aligned}
\end{equation}
\begin{equation}
    \begin{aligned}
        &\sum_{n=0}^{\infty}\frac{\left(-1\right)^nF_\beta\left(n+a,b;k\right)}{\left(n+a\right)^{2k\left(N+1\right)-1}}
        \\&=\sum_{n=0}^{\infty}{\frac{\left(-1\right)^n}{\left(n+a\right)^{2k\left(N+1\right)-1}}\left(-\frac{1}{2\pi\left(\frac{\beta\left(n+a\right)}{\pi}\right)}+\frac{1}{\pi}\sum_{i=0}^{\infty}\frac{{(-1)}^i\left(\frac{\beta\left(n+a\right)}{\pi}\right)^{2k-1}}{{(i+b)}^{2k}+\left(\frac{\beta\left(n+a\right)}{\pi}\right)^{2k}}\right)}
        \\&=\sum_{n=0}^{\infty}{\frac{\left(-1\right)^n\beta^{k-1}}{\left(n+a\right)^{2kN}}\sum_{i=0}^{\infty}\frac{{(-1)}^{i}}{\alpha^k{(i+b)}^{2k}+\beta^k{(n+a)}^{2k}}}-\frac{1}{2\beta}\sum_{n=0}^{\infty}\frac{\left(-1\right)^n}{\left(n+a\right)^{2k\left(N+1\right)}}.
    \end{aligned}
\end{equation}

For the last term of  \eqref{eq1-1}
\begin{align*}
    {\frac{\left(-1\right)^{p+1}}
        {\pi}\left(\frac{\alpha}{\pi}\right)^{2kp-1}
        \zeta_E\left(2kp,a\right)\zeta_E
        \left(2k\left(N+1-p\right),b\right)},
\end{align*}
when $p=0$, we have
\begin{align*}
    -\frac{1}{\alpha}\zeta_E\left(0,a\right)\zeta_E\left(2k(N+1),b\right),
\end{align*}
when $p=N+1$, we have
\begin{align*}
    \frac{\left(-1\right)^{N}}{\pi}\left(\frac{\alpha}{\pi}\right)^{2k(N+1)-1}\zeta_E\left(2k(N+1),a\right)\zeta_E\left(0,b\right).
\end{align*}
Since $\zeta_E\left(0,x\right)=\frac{1}{2}$,
 \eqref{eq1-1} can be written as
    \begin{align*}
        &\beta^{k\left(N+1\right)-1}
        \left(\sum_{n=0}^{\infty}\frac{{(-1)}^n\alpha^{k-1}}
        {\left(n+b\right)^{2k(N+1)-1}}\sum_{i=0}^{\infty}
        \frac{\left(-1\right)^i
        \left(n+b\right)^{2k-1}}{\alpha^k\left(n+b\right)^{2k}+
        \beta^k\left(i+a\right)^{2k}}\right)\\
        &=
        \left(-1\right)^N\alpha^{k\left(N+1\right)-1}
        \left(\sum_{n=0}^{\infty}\frac{{(-1)}^n\beta^{k-1}}
        {\left(n+a\right)^{2k(N+1)-1}}\sum_{i=0}^{\infty}
        \frac{\left(-1\right)^i
        \left(n+a\right)^{2k-1}}{\beta^k\left(n+a\right)^{2k}+
        \alpha^k\left(i+b\right)^{2k}}\right)\\
        &\quad+
        \sum_{p=1}^{N}{\left(-1\right)^{p+1}
        \zeta_E\left(2kp,a\right)\zeta_E\left(2k(N-p+1),b\right)
        \alpha^{kp-1}\beta^{k(N+1-p)-1}}.
    \end{align*}
After rearranging, we obtain (\ref{gen1.1}), the series form of Theorem \ref{res5}.

(2)
    Let
$$a_n=b_n=\left(-1\right)^{n-1},~
x_n=-\frac{\left(n-1+a\right)^{2k}}{\alpha^k},~
y_n=\frac{\left(n-1+b\right)^{2k}}{\beta^k}.$$
Substituting into  \eqref{dirichletzeta}, the corresponding Dirichlet series can be written in terms of $\zeta_E\left(2k,x\right)$ as follows:
\begin{align*}
    \zeta_{x,a}\left(p\right)
    &=\sum_{n=1}^{\infty}{\left(-1\right)^{n-1}
    \frac{1}{\left(-\frac{\left(n-1+a\right)^{2k}}
    {\alpha^k}\right)^p}}
    \\&=\left(-\alpha^k\right)^p\sum_{n=1}^{\infty}
    {\left(-1\right)^{n-1}\frac{1}{\left(n-1+a\right)^{2kp}}}
    \\&=\left(-1\right)^p\alpha^{kp}\sum_{n=0}^{\infty}
    \frac{\left(-1\right)^n}{\left(n+a\right)^{2kp}}
    \\&=\left(-1\right)^p\alpha^{kp}\zeta_E\left(2kp,a\right),
\end{align*}
\begin{align*}
    \zeta_{y,b}\left(N+1-p\right)
    =&\sum_{n=1}^{\infty}{\left(-1\right)^{n-1}
    \frac{1}{\left(\frac{\left(n-1+b\right)^{2k}}
    {\beta^k}\right)^{\left(N+1-p\right)}}}
    \\&=\left(\beta^k\right)^{\left(N+1-p\right)}
    \sum_{n=1}^{\infty}{\left(-1\right)^{n-1}
    \frac{1}{\left(n-1+b\right)^{2k\left(N+1-p\right)}}}
    \\&=\beta^{k\left(N+1-p\right)}
    \sum_{n=0}^{\infty}\frac{\left(-1\right)^n}
    {\left(n+b\right)^{2k\left(N+1-p\right)}}
    \\&=\beta^{k\left(N+1-p\right)}
    \zeta_E\left(2k\left(N+1-p\right),b\right).
\end{align*}
Furthermore, substituting ${a_n}$, ${b_n}$, ${x_n}$, ${y_n}$ into  \eqref{zetasc} yields the corresponding zeta generating functions:
\begin{align*}
\psi_{x,a}\left(z\right)
&=\sum_{i=1}^{\infty}{\left(-1\right)^{i-1}\frac{z}
{-\frac{\left(i-1+a\right)^{2k}}{\alpha^k}-z}}
\\&=-\sum_{i=0}^{\infty}{\left(-1\right)^i\frac{\alpha^kz}
{\left(i+a\right)^{2k}+\alpha^kz}},
\end{align*}
\begin{align*}
\psi_{y,b}\left(z\right)
&=\sum_{i=1}^{\infty}{\left(-1\right)^{i-1}\frac{z}
{\frac{\left(i-1+b\right)^{2k}}{\beta^k}-z}}
\\&=\sum_{i=0}^{\infty}{\left(-1\right)^i\frac{\beta^kz}
{\left(i+b\right)^{2k}-\beta^kz}}.
\end{align*}
Substituting ${x_n}$, ${y_n}$ gives
\begin{align*}
    \psi_{x,a}\left(y_n\right)=-\sum_{i=0}^{\infty}
    {\left(-1\right)^i\frac{\alpha^k\left(n-1+b\right)^{2k}}
    {\beta^k\left(i+a\right)^{2k}+\alpha^k\left(n-1+b\right)^{2k}}},
\end{align*}
\begin{align*}
    \psi_{y,b}\left(x_n\right)=-\sum_{i=0}^{\infty}
    {\left(-1\right)^i\frac{\beta^k\left(n-1+a\right)^{2k}}
    {\alpha^k\left(i+b\right)^{2k}+\beta^k\left(n-1+a\right)^{2k}}}.
\end{align*}
Therefore, applying Lemma \ref{Dirichlet_jishu}, we have
\begin{align*}
&\sum_{p=1}^{N}{\left(-1\right)^p\alpha^{kp}
\beta^{k\left(N+1-p\right)}\zeta_E\left(2kp,a\right)
\zeta_E\left(2k\left(N+1-p\right),b\right)}
\\&=\sum_{n=1}^{\infty}\frac{\left(-1\right)^{n-1}}
{\left(\frac{\left(n-1+b\right)^{2k}}{\beta^k}\right)^{N+1}}
\left(-\sum_{i=0}^{\infty}{\left(-1\right)^i
\frac{\alpha^k\left(n-1+b\right)^{2k}}{\beta^k\left(i+a\right)^{2k}
+\alpha^k\left(n-1+b\right)^{2k}}}\right)
\\
&\quad+\sum_{n=1}^{\infty}\frac{\left(-1\right)^{n-1}}
{\left(-\frac{\left(n-1+a\right)^{2k}}{\alpha^k}\right)^{N+1}}
\left(-\sum_{i=0}^{\infty}{\left(-1\right)^i
\frac{\beta^k\left(n-1+a\right)^{2k}}{\alpha^k\left(i+b\right)^{2k}
+\beta^k\left(n-1+a\right)^{2k}}}\right)
\\&=-\beta^{k\left(N+1\right)}\sum_{n=0}^{\infty}
\frac{\left(-1\right)^n}{\left(n+b\right)^{2kN}}
\sum_{i=0}^{\infty}\frac{\left(-1\right)^i\alpha^k}
{\beta^k\left(i+a\right)^{2k}+\alpha^k\left(n+b\right)^{2k}}
\\
&\quad+\left(-1\right)^N\alpha^{k\left(N+1\right)}
\sum_{n=0}^{\infty}\frac{\left(-1\right)^n}{\left(n+a\right)^{2kN}}
\sum_{i=0}^{\infty}\frac{\left(-1\right)^i\beta^k}
{\alpha^k\left(i+b\right)^{2k}+\beta^k\left(n+a\right)^{2k}}.
\end{align*}
This equation is equivalent to the result of Theorem \ref{res5}
\qed

\subsection*{\bf Proof of Theorem \ref{res6}}
We give two proofs.

(1) Since $\zeta_{E}(z,x)$ converges absolutely for $\operatorname{Re}(z)>0$, substituting the integral definition of the alternating odd-order Hurwitz kernel into the series on the left-hand side of the theorem and interchanging the infinite sum and the complex integral yields
    \begin{equation}
        \begin{aligned}\label{dingyijifen4_1}
            &\sum_{n=0}^{\infty}\frac{\left(-1\right)^n
            G_\alpha\left(n+b,a;k\right)}{\left(n+b\right)^{2km+1}}
            \\&=\sum_{n=0}^{\infty}\frac{\left(-1\right)^n}
            {\left(n+b\right)^{2km+1}}\left(\frac{1}{2i\pi}
            \int_{\left(c\right)}{\frac{\zeta_E\left(1-s,a\right)}
            {2k \sin{\left(\frac{\pi s}{2k}\right)}}
            \left(\frac{\alpha\left(n+b\right)}{\pi}\right)^{-s}{\rm d}s}\right)
            \\&=\frac{1}{2i\pi}\int_{\left(c\right)}\frac{\zeta_E\left(1-s,a\right)
            \zeta_E\left(2km+1+s,b\right)}{2k \sin{\left(\frac{\pi s}{2k}\right)}}
            \left(\frac{\alpha}{\pi}\right)^{-s}{\rm d}s.
        \end{aligned}
    \end{equation}
    We now evaluate this integral by shifting the line of integration. In the complex plane, consider the rectangular region bounded by the four line segments
    $[c-iT,c+iT]$, $[c+iT,d+iT]$, $[d+iT,d-iT]$, $[d-iT,c-iT]$, where $d=-2km-c$.
    Because the alternating Hurwitz zeta function is analytic in the entire complex plane, all poles of this integrand come from the denominator
    $2k \sin{\left(\frac{\pi s}{2k}\right)}$.
    The poles are at integers $-2kp$ ($p\in\left\{0,1,\ldots,m\right\}$), all of which are simple poles.
    The residues at these poles are
\begin{equation}
    \begin{aligned}
        R_{-2kp}
&= \lim_{s\rightarrow -2kp}
\frac{(s+2kp)\zeta_E(1-s,a)\zeta_E(2km+1+s,b)}{2k \sin\left(\frac{\pi s}{2k}\right)}
\left(\frac{\alpha}{\pi}\right)^{-s} \\
&= \frac{1}{\pi}(-1)^p \zeta_E(1+2kp,a)\zeta_E(2km-2kp+1,b)
\left(\frac{\alpha}{\pi}\right)^{2kp}.
    \end{aligned}
\end{equation}
    By Cauchy's residue theorem, we have
    \begin{equation}
    \begin{aligned}\label{cauchy4}
        &\frac{1}{2i\pi}\left[\int_{c-iT}^{c+iT}+\int_{c+iT}^{d+iT}
        +\int_{d+iT}^{d-iT}+\int_{d-iT}^{c-iT}\right]
        \frac{\zeta_E\left(1-s,a\right)\zeta_E\left(s+2km-1,b\right)}
        {2k \sin{\left(\frac{\pi s}{2k}\right)}}\left(\frac{\alpha}{\pi}\right)^{-s}{\rm d}s
        \\&=\frac{1}{\pi}(-1)^p \zeta_E(1+2kp,a)\zeta_E(2km-2kp+1,b)
\left(\frac{\alpha}{\pi}\right)^{2kp}.
    \end{aligned}
    \end{equation}
    Using Lemma \ref{constants} and  \eqref{sin_xiao}, as $T\rightarrow\infty$, the integrals along the horizontal segments tend to $0$. Hence,
    \begin{equation}
        \begin{aligned}\label{0_4}
        &\int_{\left(c\right)}{\frac{\zeta_E\left(1-s,a\right)
        \zeta_E\left(s+2km+1,b\right)}{2k \sin{\left(\frac{\pi s}{2k}\right)}}
        \left(\frac{\alpha}{\pi}\right)^{-s}{\rm d}s}
        \\&=\int_{\left(d\right)}{\frac{\zeta_E\left(1-s,a\right)
        \zeta_E\left(s+2km+1,b\right)}{2k \sin{\left(\frac{\pi s}{2k}\right)}}
        \left(\frac{\alpha}{\pi}\right)^{-s}{\rm d}s}+\sum_{p=0}^{m}R_{-2kp}.
    \end{aligned}
    \end{equation}
    Next, making the substitution $s\rightarrow-s-2km$ yields
    \begin{equation}
        \begin{aligned}\label{tihuan4}
        &\int_{\left(d\right)}{\frac{\zeta_E\left(1-s,a\right)
        \zeta_E\left(s+2km+1,b\right)}{2k \sin{\left(\frac{\pi s}{2k}\right)}}
        \left(\frac{\alpha}{\pi}\right)^{-s}{\rm d}s}
        \\&=\left(\frac{\alpha}{\pi}\right)^{2km}\left(-1\right)^{m+1}
        \int_{\left(c\right)}{\frac{\zeta_E\left(1-s,b\right)
        \zeta_E\left(s+2km+1,a\right)}{2k \sin{\left(\frac{\pi s}{2k}\right)}}
        \left(\frac{\beta}{\pi}\right)^{-s}{\rm d}s}.
    \end{aligned}
    \end{equation}
    Combining  \eqref{0_4} and \eqref{tihuan4}, we have
    \begin{equation}
        \begin{aligned}
        &\frac{1}{2i\pi}\int_{\left(c\right)}{\frac{\zeta_E\left(1-s,a\right)
        \zeta_E\left(s+2km-1,b\right)}{2k \sin{\left(\frac{\pi s}{2k}\right)}}
        \left(\frac{\alpha}{\pi}\right)^{-s}{\rm d}s}
        \\&=\left(\frac{\alpha}{\pi}\right)^{2km}\left(-1\right)^{m+1}
        \frac{1}{2i\pi}\int_{\left(c\right)}{\frac{\zeta_E\left(1-s,b\right)
        \zeta_E\left(s+2km+1,a\right)}{2k \sin{\left(\frac{\pi s}{2k}\right)}}
        \left(\frac{\beta}{\pi}\right)^{-s}{\rm d}s}
        \\
        &\quad+\frac{1}{\pi}\sum_{p=0}^{m}(-1)^p \zeta_E(1+2kp,a)\zeta_E(2km-2kp+1,b)
\left(\frac{\alpha}{\pi}\right)^{2kp}.
    \end{aligned}
    \end{equation}
    Recalling  \eqref{dingyijifen4_1}, we obtain
    \begin{equation}
        \begin{aligned}\label{res6_1}
        &\sum_{n=0}^{\infty}\frac{\left(-1\right)^nG_\alpha\left(n+b,a;k\right)}
        {\left(n+b\right)^{2km+1}}
        =\left(\frac{\alpha}{\pi}\right)^{2km}\left(-1\right)^{m+1}
        \sum_{n=0}^{\infty}\frac{\left(-1\right)^nG_\beta(n+a,b;k)}
        {{(n+a)}^{2km+1}}\\
        &\quad+\frac{1}{\pi}\sum_{p=0}^{m}{\left(-1\right)^p
        \zeta_E\left(1+2kp,a\right)\zeta_E\left(2km-2kp+1,b\right)
        \left(\frac{\alpha}{\pi}\right)^{2kp}}.
    \end{aligned}
    \end{equation}
 Since $\alpha\beta=\pi^2$, by multiplying both sides by $\beta^{km}$ and rearranging, we first have (\ref{ra-ty-G}), the functional form of Theorem \ref{res6}.

 Then in  \eqref{res6_1}, substituting the series expressions for $G_\alpha(n+b,a;k)$ and $G_\beta(n+a,b;k)$ (see Definition \ref{alternating_hurwitz_kernel_odd}) yields
\begin{equation}
    \begin{aligned}
        &\sum_{n=0}^{\infty}\frac{\left(-1\right)^n
        G_\alpha\left(n+b,a;k\right)}{\left(n+b\right)^{2km+1}}
        \\&=\sum_{n=0}^{\infty}\frac{\left(-1\right)^n}
        {\left(n+b\right)^{2km+1}}\left(\frac{-\widetilde{\psi}\left(a\right)}{\pi}
        -\frac{\left(\frac{\alpha\left(n+b\right)}{\pi}\right)^{2k}}{\pi}
        \sum_{i=0}^{\infty}\frac{\left(-1\right)^i}{\left(i+a\right)
        \left(\left(i+a\right)^{2k}+\left(\frac{\alpha\left(n+b\right)}{\pi}\right)^{2k}\right)}\right)
        \\&=-\frac{1}{\pi}
        \sum_{n=0}^{\infty}\frac{\left(-1\right)^n\alpha^k}{\left(n+b\right)^{2k(m-1)+1}}
        \sum_{i=0}^{\infty}\frac{\left(-1\right)^i}{\left(i+a\right)
        \left(\beta^k\left(i+a\right)^{2k}+\alpha^k\left(n+b\right)^{2k}\right)}
        \\
        &\quad-\frac{1}{\pi}\sum_{n=0}^{\infty}\frac{\left(-1\right)^n
        \widetilde{\psi}(a)}{\left(n+b\right)^{2km+1}},
    \end{aligned}
\end{equation}
\begin{equation}
    \begin{aligned}
        &\sum_{n=0}^{\infty}\frac{\left(-1\right)^n
        G_\beta\left(n+a,b;k\right)}{\left(n+a\right)^{2km+1}}
        =-\frac{1}{\pi}\sum_{n=0}^{\infty}\frac{\left(-1\right)^n
        \widetilde{\psi}(b)}{\left(n+a\right)^{2km+1}}\\&-
        \frac{1}{\pi}\sum_{n=0}^{\infty}\frac{\left(-1\right)^n\beta^k}
        {\left(n+a\right)^{2k(m-1)+1}}\sum_{i=0}^{\infty}
        \frac{\left(-1\right)^i}{\left(i+b\right)
        \left(\alpha^k\left(i+b\right)^{2k}+\beta^k\left(n+a\right)^{2k}\right)}.
    \end{aligned}
\end{equation}
For the last term of  \eqref{ra-ty-G}
    \begin{align*}
        \left(-1\right)^p\zeta_E\left(2kp+1,a\right)
        \zeta_E\left(2k\left(m-p\right)+1,b\right)\alpha^{kp}\beta^{k\left(m-p\right)},
    \end{align*}
when $p=0$, we have
\begin{align*}
    \zeta_E\left(1,a\right)\zeta_E\left(2km+1,b\right)\beta^{km},
\end{align*}
when $p=m$, we have
\begin{align*}
    \left(-1\right)^m\zeta_E\left(2km+1,a\right)\zeta_E\left(1,b\right)\alpha^{km}.
\end{align*}
Since $\widetilde{\psi}(x)=-\zeta_E\left(1,x\right)$,  \eqref{ra-ty-G} can be written as
\begin{equation}\label{res6_2}
    \begin{aligned}
        &(-1)^m\alpha^{km}\sum_{n=0}^{\infty}
        \frac{\left(-1\right)^n\beta^k}{\left(n+a\right)^{2k(m-1)+1}}
        \sum_{i=0}^{\infty}\frac{\left(-1\right)^i}{\left(i+b\right)
        \left(\alpha^k\left(i+b\right)^{2k}+\beta^k\left(n+a\right)^{2k}\right)}
        \\&=\beta^{km}\sum_{n=0}^{\infty}
        \frac{\left(-1\right)^n\alpha^k}{\left(n+b\right)^{2k(m-1)+1}}
        \sum_{i=0}^{\infty}\frac{\left(-1\right)^i}{\left(i+a\right)
        \left(\beta^k\left(i+a\right)^{2k}+\alpha^k\left(n+b\right)^{2k}\right)}
        \\
        &\quad+\sum_{p=1}^{m-1}{\left(-1\right)^p\zeta_E\left(2kp+1,a\right)
        \zeta_E\left(2k\left(m-p\right)+1,b\right)\alpha^{kp}\beta^{k\left(m-p\right)}}.
    \end{aligned}
\end{equation}
After rearranging, we obtain (\ref{Th3}), the series form of Theorem \ref{res6}.

(2) Let
$$a_n=\frac{\left(-1\right)^{n-1}}{n-1+a},~
b_n=\frac{\left(-1\right)^{n-1}}{n-1+b},~
x_n=-\frac{{(n-1+a)}^{2k}}{\alpha^k},~
y_n=\frac{{(n-1+b)}^{2k}}{\beta^k}.$$
Substituting into  \eqref{dirichletzeta}, the corresponding Dirichlet series can be written in terms of $\zeta_E\left(2k+1,x\right)$ as follows:
\begin{align*}
    \zeta_{x,a}\left(p\right)
&= \sum_{n=1}^{\infty}
    \frac{(-1)^{n-1}}{(n-1+a)} \frac{1}{\left(-\frac{(n-1+a)^{2k}}{\alpha^k}\right)^p} \\
&= \left(-\alpha^k\right)^p \sum_{n=1}^{\infty} \frac{(-1)^{n-1}}{(n-1+a)^{2kp+1}} \\
&= (-1)^p \alpha^{kp} \sum_{n=0}^{\infty} \frac{(-1)^n}{(n+a)^{2kp+1}} \\
&= (-1)^p \alpha^{kp} \zeta_E\left(2kp+1,a\right),
\end{align*}
\begin{align*}
    \zeta_{y,b}\left(m-p\right)&=\sum_{n=1}^{\infty}
    {\frac{\left(-1\right)^{n-1}}{(n-1+b)}\frac{1}{\left(\frac{{(n-1+b)}^{2k}}{\beta^k}\right)^{m-p}}}
    \\&=\beta^{k(m-p)}\sum_{n=1}^{\infty}\frac{\left(-1\right)^{n-1}}{{(n-1+b)}^{2k(m-p)+1}}
    \\&=\beta^{k(m-p)}\sum_{n=0}^{\infty}\frac{\left(-1\right)^n}{{(n+b)}^{2k(m-p)+1}}
    \\&=\beta^{k(m-p)}\zeta_E\left(2k(m-p)+1,b\right).
\end{align*}
Furthermore, substituting ${a_n}$, ${b_n}$, ${x_n}$, ${y_n}$ into  \eqref{zetasc} yields the corresponding zeta generating functions:
\begin{align*}
\psi_{x,a}\left(z\right)&=\sum_{i=1}^{\infty}{\frac{\left(-1\right)^{i-1}}
{(i-1+a)}\frac{z}{\left(-\frac{{(i-1+a)}^{2k}}{\alpha^k}-z\right)}}
\\&=-\sum_{i=0}^{\infty}\frac{\left(-1\right)^i\alpha^kz}{(i+a)({(i+a)}^{2k}+\alpha^kz)},
\end{align*}
\begin{align*}
\psi_{y,b}\left(z\right)&=\sum_{i=1}^{\infty}{\frac{\left(-1\right)^{i-1}}
{(i-1+b)}\frac{z}{\left(\frac{{(i-1+b)}^{2k}}{\beta^k}-z\right)}}
\\&=\sum_{i=0}^{\infty}\frac{\left(-1\right)^i\beta^kz}{(i+b)({(i+b)}^{2k}-\beta^kz)}.
\end{align*}  
Substituting ${x_n}$, ${y_n}$ gives
\begin{align*}
    \psi_{x,a}\left(y_n\right)=-\sum_{i=0}^{\infty}\frac{\left(-1\right)^i
    \alpha^k{(n-1+b)}^{2k}}{(i+a)({\beta^k(i+a)}^{2k}+\alpha^k{(n-1+b)}^{2k})},
\end{align*}
\begin{align*}
    \psi_{y,b}\left(x_n\right)=\sum_{i=0}^{\infty}\frac{\left(-1\right)^i
    \beta^k{(n-1+a)}^{2k}}{(i+b)({\alpha^k(i+b)}^{2k}+\beta^k{(n-1+a)}^{2k})}.
\end{align*}
Therefore, applying Lemma \ref{Dirichlet_jishu}, we have
\begin{align*}
&\sum_{p=1}^{m-1}{\left(-1\right)^p\alpha^{kp}\zeta_E\left(2kp+1,a\right)}
\beta^{k(m-p)}\zeta_E\left(2k(m-p)+1,b\right)\\
&=\sum_{n=1}^{\infty}
\frac{\frac{\left(-1\right)^{n-1}}{n-1+a}}{\left(-\frac{{(n-1+a)}^{2k}}
{\alpha^k}\right)^m}\left(\sum_{i=0}^{\infty}\frac{\left(-1\right)^i
\beta^k{(n-1+a)}^{2k}}{(i+b)({\alpha^k(i+b)}^{2k}+\beta^k{(n-1+a)}^{2k})}\right)
\\
&\quad+\sum_{n=1}^{\infty}\frac{\frac{\left(-1\right)^{n-1}}{n-1+b}}
{\left(\frac{{(n-1+b)}^{2k}}{\beta^k}\right)^m}\left(-\sum_{i=0}^{\infty}
\frac{\left(-1\right)^i\alpha^k{(n-1+b)}^{2k}}{(i+a)({\beta^k(i+a)}^{2k}
+\alpha^k{(n-1+b)}^{2k})}\right)\\
&= (-\alpha^k)^m\beta^k\sum_{n=0}^{\infty}
\frac{\left(-1\right)^n}{\left(n+a\right)^{2k(m-1)+1}}\sum_{i=0}^{\infty}
\frac{\left(-1\right)^i}{(i+b)({\alpha^k(i+b)}^{2k}+\beta^k{(n+a)}^{2k})}
\\
&\quad-\alpha^k\beta^{km}\sum_{n=0}^{\infty}\frac{\left(-1\right)^n}
{\left(n+b\right)^{2k\left(m-1\right)+1}}\sum_{i=0}^{\infty}
\frac{\left(-1\right)^i}{(i+a)({\beta^k(i+a)}^{2k}+\alpha^k{(n+b)}^{2k})}.
\end{align*}
This equation is equivalent to \eqref{res6_2}.
\qed

\begin{remark}
   \textup{Theorems \ref{res5} and \ref{res6} correspond to Ramanujan-type identities for the alternating even-order and odd-order Hurwitz kernels under the modular symmetry $\alpha\beta=\pi^2$, respectively, and share similar series expressions.}
\end{remark}

\subsection*{\bf Proof of Theorem \ref{res7}}
When $2km+2<d<2km+3$, by Cauchy's residue theorem we have
\begin{align*}
    \frac{1}{2i\pi}\left[\int_{c-iT}^{c+iT}+\int_{c+iT}^{d+iT}+\int_{d+iT}^{d-iT}+\int_{d-iT}^{c-iT}\right]
\frac{\zeta_E\left(1-s,a\right)}{2k\cos{\left(\frac{\pi\left(s+k-1\right)}{2k}\right)}}x^{-s}{\rm d}s
=\sum_{p=0}^{m}R_{2kp+1}.
\end{align*}
The integrand has simple poles at integers $s=2kp+1~(p\in\left\{0,1,\ldots,m\right\})$.
For a simple pole $s=2kp+1$, the residue is computed as
    \begin{align*}
        R_{2kp+1}&=\lim_{s\rightarrow 2kp+1}
        {\frac{\left(s-\left(2kp+1\right)\right)
        \zeta_E\left(1-s,a\right)}{2k \cos{\left(\frac{\pi\left(s+k-1\right)}{2k}\right)}}x^{-s}}
        \\&=\frac{\zeta_E\left(-2kp,a\right)}{-\pi s i n{\left(p\pi+\frac{\pi}{2}\right)}}x^{-2kp-1}
        \\&=\frac{(-1)^{p+1}}{2\pi} E_{2kp}(a) x^{-2kp-1}.
    \end{align*}
Hence,
    \begin{align*}
    &\frac{1}{2i\pi}\left[\int_{c-iT}^{c+iT}+\int_{c+iT}^{d+iT}
    +\int_{d+iT}^{d-iT}+\int_{d-iT}^{c-iT}\right]
    \frac{\zeta_E\left(1-s,a\right)}{2k \cos{\left(\frac{\pi\left(s+k-1\right)}
    {2k}\right)}}x^{-s}{\rm d}s\\&=\frac{1}{2\pi} \sum_{p=0}^{m} (-1)^{p+1} E_{2kp}(a) x^{-2kp-1}.
    \end{align*}
Since the alternating Hurwitz zeta function converges absolutely for $\operatorname{Re}(z)>0$, we have
    \begin{equation}\label{jifen2}
        \begin{aligned}
            &\sum_{n=0}^{\infty}{\left(-1\right)^n\left(n+b\right)^{2km+1}
            \left(F_\alpha(n+b,a;k)-\frac{1}{2\pi} \sum_{p=1}^{m} (-1)^{p+1} E_{2kp}(a) \left(\frac{\alpha}{\pi}\right)^{-2kp-1}\right)}
            \\&=\frac{1}{2i\pi}\int_{\left(d\right)}{\frac{\zeta_E\left(1-s,a\right)\zeta_E\left(s-2km-1,b\right)}{2k \cos{\left(\frac{\pi\left(s+k-1\right)}{2k}\right)}}\left(\frac{\alpha}{\pi}\right)^{-s}{\rm d}s}.
        \end{aligned}
    \end{equation}
    We now construct a rectangular contour and evaluate the integral by shifting the line of integration. In the complex plane, consider the rectangular region bounded by the four line segments 
    $[d-iT,d+iT]$, $[d+iT,e+iT]$, $[e+iT,e-iT]$, $[e-iT,d-iT]$, where $e=2km+2-d$.
    Since the alternating Hurwitz zeta function is analytic in the entire complex plane, all poles of this integrand come from the denominator
    $2k \cos{\left(\frac{\pi\left(s+k-1\right)}{2k}\right)}$. The poles are at integers $s=2kp+1~(p\in\left\{0,1,\ldots,m\right\})$, all of which are simple poles.
    For a simple pole $s=2kp+1$, the residue is computed as
    \begin{align*}
        R_{2kp+1}&=\lim_{s\rightarrow 2kp+1}
        {\frac{\left(s-\left(2kp+1\right)\right)\zeta_E\left(1-s,a\right)
        \zeta_E\left(s-2km-1,b\right)}{2k \cos{\left(\frac{\pi\left(s+k-1\right)}
        {2k}\right)}}\left(\frac{\alpha}{\pi}\right)^{-s}}
        \\&=\frac{\zeta_E\left(-2kp,a\right)\zeta_E\left(2kp-2km,b\right)}
        {-\pi s i n{\left(p\pi+\frac{\pi}{2}\right)}}
        \left(\frac{\pi}{\alpha}\right)^{2kp+1}
        \\&=\frac{(-1)^{p+1}}{4\pi} E_{2kp}(a) E_{2k(m-p)}(b) \left(\frac{\pi}{\alpha}\right)^{2kp+1}.
    \end{align*}
    By Cauchy's residue theorem, we have
    \begin{equation}\label{cauchy2}
    \begin{aligned}
        &\frac{1}{2i\pi}\left[\int_{d-iT}^{d+iT}+\int_{d+iT}^{e+iT}
        +\int_{e+iT}^{e-iT}+\int_{e-iT}^{d-iT}\right]\frac{\zeta_E\left(1-s,a\right)
        \zeta_E\left(s-2km-1,b\right)}{2k \cos{\left(\frac{\pi\left(s+k-1\right)}{2k}\right)}}
        \left(\frac{\alpha}{\pi}\right)^{-s}{\rm d}s
        \\&=\frac{1}{4\pi} \sum_{p=0}^{m} (-1)^{p+1} E_{2kp}(a) E_{2k(m-p)}(b) \left(\frac{\pi}{\alpha}\right)^{2kp+1}.
    \end{aligned}
    \end{equation}
    Using Lemma \ref{constants} and  \eqref{cos_xiao}, 
    as $T\rightarrow\infty$, the integrals along the horizontal segments tend to $0$. After the substitution $s\rightarrow2km+2-s$, we obtain
    \begin{equation}\label{tihuan2}
    \begin{aligned}
        &\int_{\left(e\right)}{\frac{\zeta_E\left(1-s,a\right)
        \zeta_E\left(s-2km-1,b\right)}{2k \cos{\left(\frac{\pi\left(s+k-1\right)}{2k}\right)}}
        \left(\frac{\alpha}{\pi}\right)^{-s}{\rm d}s}
        \\&=\left(\frac{\alpha}{\pi}\right)^{-2km-2}\left(-1\right)^{m-1}\int_{\left(d\right)}
        {\frac{\zeta_E\left(1-s,b\right)\zeta_E\left(s-2km-1,a\right)}
        {2k \cos{\left(\frac{\pi\left(s+k-1\right)}{2k}\right)}}
        \left(\frac{\beta}{\pi}\right)^{-s}{\rm d}s}.
    \end{aligned}
    \end{equation}
Combining  \eqref{cauchy2} and \eqref{tihuan2}, we have
    \begin{equation}
    \begin{aligned}
        &\frac{1}{2i\pi}\int_{\left(d\right)}{\frac{\zeta_E\left(1-s,a\right)
        \zeta_E\left(s-2km-1,b\right)}{2k \cos{\left(\frac{\pi\left(s+k-1\right)}
        {2k}\right)}}\left(\frac{\alpha}{\pi}\right)^{-s}{\rm d}s}
        \\&=\left(\frac{\alpha}{\pi}\right)^{-2km-2}\left(-1\right)^{m-1}\frac{1}{2i\pi}
        \int_{\left(d\right)}{\frac{\zeta_E\left(1-s,b\right)
        \zeta_E\left(s-2km-1,a\right)}{2k \cos{\left(\frac{\pi\left(s+k-1\right)}
        {2k}\right)}}\left(\frac{\beta}{\pi}\right)^{-s}{\rm d}s}
        \\
        &\quad+\frac{1}{4\pi} \sum_{p=0}^{m} (-1)^{p+1} E_{2kp}(a) E_{2k(m-p)}(b) \left(\frac{\pi}{\alpha}\right)^{2kp+1}.
    \end{aligned}
    \end{equation}
Recalling  \eqref{jifen2}, we obtain
    \begin{equation}
    \begin{aligned}
        &\sum_{n=0}^{\infty}{\left(-1\right)^n\left(n+b\right)^{2km+1}
        \left(F_\alpha(n+b,a;k)-\sum_{p=1}^{m}\frac{\left(-1\right)^n
        E_{2kp}(a)}{2\pi}\left(\frac{\pi}{\alpha}\right)^{2kp+1}\right)}
        \\&=\left(\frac{\alpha}{\pi}\right)^{-2km-2}\left(-1\right)^{m-1}
        \sum_{n=0}^{\infty}{\left(-1\right)^n{(n+a)}^{2km+1}
        \left(F_\beta(n+a,b;k)-\sum_{p=1}^{m}\frac{\left(-1\right)^n
        E_{2kp}(b)}{2\pi}\left(\frac{\pi}{\beta}\right)^{2kp+1}\right)}
        \\
        &\quad+\frac{1}{4\pi} \sum_{p=0}^{m} (-1)^{p+1} E_{2kp}(a) E_{2k(m-p)}(b) \left(\frac{\pi}{\alpha}\right)^{2kp+1}.
    \end{aligned}
    \end{equation}
Since $\alpha\beta=\pi^2$, by multiplying both sides by $\alpha^{km+1}$ and rearranging, we have (\ref{Th4}).
\qed

\subsection*{\bf Proof of Theorem \ref{res8}}
When $2km<d<2km+1$, by Cauchy's residue theorem we have
\begin{align*}
    \frac{1}{2i\pi}\int_{\left(c\right)}{\frac{\zeta_E\left(1-s,a\right)}
    {2k \sin{\left(\frac{\pi s}{2k}\right)}}x^{-s}{\rm d}s}
    =\frac{1}{2i\pi}\int_{\left(d\right)}{\frac{\zeta_E\left(1-s,a\right)}
    {2k \sin{\left(\frac{\pi s}{2k}\right)}}x^{-s}{\rm d}s}
    +\sum_{p=1}^{m}R_{2kp}.
\end{align*}
The integrand has simple poles at integers $s=2kp~(p\in\left\{0,1,\ldots,m\right\})$.
For a simple pole $s=2kp$, the residue is computed as
    \begin{align*}
    R_{2kp}&=\lim_{s\rightarrow\left(2kp\right)}
    {\frac{\left(s-2kp\right)\zeta_E\left(1-s,a\right)}
    {2k \sin{\left(\frac{\pi s}{2k}\right)}}x^{-s}}
    \\&=\frac{\zeta_E\left(1-2kp,a\right)}{\pi c o s{(p\pi)}}x^{-2kp}
    \\&=\frac{\left(-1\right)^p}{2\pi}E_{2kp-1}(a)x^{-2kp}.
    \end{align*}
Hence,
\begin{align*}
    \frac{1}{2i\pi}\int_{\left(c\right)}
    {\frac{\zeta_E\left(1-s,a\right)}{2k \sin{\left(\frac{\pi s}{2k}\right)}}x^{-s}{\rm d}s}
    =\frac{1}{2i\pi}\int_{\left(d\right)}{\frac{\zeta_E\left(1-s,a\right)}
    {2k \sin{\left(\frac{\pi s}{2k}\right)}}x^{-s}{\rm d}s}
    +\frac{1}{2\pi} \sum_{p=0}^{m} (-1)^p E_{2kp-1}(a) x^{-2kp},
\end{align*}
where $E_{-1}(a)$ is interpreted so that
$$E_{-1}(a)=2\zeta_E(1,a).$$
Since the alternating Hurwitz zeta function $\zeta_E(z,x)$ converges absolutely for $\operatorname{Re}(z)>0$, we have
    \begin{equation}\label{jifen3}
        \begin{aligned}
            &\sum_{n=0}^{\infty}{\left(-1\right)^n\left(n+b\right)
            ^{2km-1}\left(G_\alpha(n+b,a;k)-\frac{1}{2\pi} \sum_{p=1}^{m} (-1)^p E_{2kp-1}(a) x^{-2kp}\right)}
            \\&=\frac{1}{2i\pi}\int_{\left(d\right)}
            {\frac{\zeta_E\left(1-s,a\right)\zeta_E\left(s-2km+1,b\right)}
            {2k \sin{\left(\frac{\pi s}{2k}\right)}}
            \left(\frac{\alpha}{\pi}\right)^{-s}{\rm d}s}.
        \end{aligned}
    \end{equation}
    We now construct a rectangular contour and evaluate the integral by shifting the line of integration. In the complex plane, consider the rectangular region bounded by the four line segments 
    $[d-iT,d+iT]$, $[d+iT,e+iT]$, $[e+iT,e-iT]$, $[e-iT,d-iT]$, where $e=2km-d$.
    Since the alternating Hurwitz zeta function is analytic in the entire complex plane, all poles of this integrand come from the denominator
    $2k \sin{\left(\frac{\pi s}{2k}\right)}$. The poles are at integers $2kp~(p\in\left\{0,1,\ldots,m\right\})$, all of which are simple poles.
    For a simple pole $s=2kp$, the residue is computed as
    \begin{align*}
        R_{2kp}&=\lim_{s\rightarrow 2kp}
        {\frac{\left(s-2kp\right)\zeta_E\left(1-s,a\right)
        \zeta_E\left(s-2km+1,b\right)}{2k \sin{\left(\frac{\pi s}{2k}\right)}}
        \left(\frac{\alpha}{\pi}\right)^{-s}}
        \\&=\frac{\zeta_E\left(1-2kp,a\right)\zeta_E\left(2kp-2km+1,b\right)}
        {\pi c o s{\left(p\pi\right)}}\left(\frac{\pi}{\alpha}\right)^{2kp}
        \\&=\frac{(-1)^p}{4\pi} E_{2kp-1}(a) E_{2k(m-p)-1}(b) \left(\frac{\pi}{\alpha}\right)^{2kp}.
    \end{align*}
    By Cauchy's residue theorem, we have
    \begin{equation}\label{cauchy3}
    \begin{aligned}
        &\frac{1}{2i\pi}\left[\int_{d-iT}^{d+iT}+\int_{d+iT}^{e+iT}
        +\int_{e+iT}^{e-iT}+\int_{e-iT}^{d-iT}\right]
        \frac{\zeta_E\left(1-s,a\right)\zeta_E\left(s-2km-1,b\right)}
        {2k \sin{\left(\frac{\pi s}{2k}\right)}}\left(\frac{\alpha}{\pi}\right)^{-s}{\rm d}s
        \\&=\frac{1}{4\pi} \sum_{p=0}^{m} (-1)^p E_{2kp-1}(a) E_{2k(m-p)-1}(b) \left(\frac{\pi}{\alpha}\right)^{2kp}.
    \end{aligned}
    \end{equation}
    Using Lemma \ref{constants} and  \eqref{sin_xiao}, 
    as $T\rightarrow\infty$, the integrals along the horizontal segments tend to $0$. After the substitution $s\rightarrow2km-s$, we obtain
    \begin{equation}\label{tihuan3}
    \begin{aligned}
        &\int_{\left(e\right)}{\frac{\zeta_E\left(1-s,a\right)
        \zeta_E\left(s-2km+1,b\right)}{2k \sin{\left(\frac{\pi s}{2k}\right)}}
        \left(\frac{\alpha}{\pi}\right)^{-s}{\rm d}s}
        \\&=\left(\frac{\alpha}{\pi}\right)^{-2km}\left(-1\right)^{m-1}
        \int_{\left(d\right)}{\frac{\zeta_E\left(1-s,b\right)
        \zeta_E\left(s-2km+1,a\right)}{2k \sin{\left(\frac{\pi s}{2k}\right)}}
        \left(\frac{\beta}{\pi}\right)^{-s}{\rm d}s}.
    \end{aligned}
    \end{equation}
    Combining  \eqref{cauchy3} and \eqref{tihuan3}, we have
    \begin{equation}
    \begin{aligned}
        &\frac{1}{2i\pi}\int_{\left(d\right)}{\frac{\zeta_E\left(1-s,a\right)
        \zeta_E\left(s-2km-1,b\right)}{2k \sin{\left(\frac{\pi s}{2k}\right)}}
        \left(\frac{\alpha}{\pi}\right)^{-s}{\rm d}s}
        \\&=\frac{1}{2i\pi}\left(\frac{\alpha}{\pi}\right)^{-2km}\left(-1\right)^{m-1}
        \int_{\left(d\right)}{\frac{\zeta_E\left(1-s,b\right)
        \zeta_E\left(s-2km-1,a\right)}{2k \sin{\left(\frac{\pi s}{2k}\right)}}
        \left(\frac{\beta}{\pi}\right)^{-s}{\rm d}s}
        \\
        &\quad+\frac{1}{4\pi} \sum_{p=0}^{m} (-1)^p E_{2kp-1}(a) E_{2k(m-p)-1}(b) \left(\frac{\pi}{\alpha}\right)^{2kp}.
    \end{aligned}
    \end{equation}
    Recalling  \eqref{jifen3}, we obtain
    \begin{equation}
    \begin{aligned}
        &\sum_{n=0}^{\infty}{\left(-1\right)^n\left(n+b\right)^{2km-1}
        \left(G_\alpha(n+b,a;k)-\sum_{p=0}^{m}\frac{\left(-1\right)^n
        E_{2kp-1}(a)}{2\pi}\left(\frac{\pi}{\alpha}\right)^{-2kp}\right)}
        \\&=\left(\frac{\alpha}{\pi}\right)^{-2km}\left(-1\right)^{m-1}
        \sum_{n=0}^{\infty}{\left(-1\right)^n{(n+a)}^{2km-1}
        \left(G_\beta(n+a,b;k)-\sum_{p=0}^{m}\frac{\left(-1\right)^n
        E_{2kp-1}(b)}{2\pi}\left(\frac{\pi}{\beta}\right)^{-2kp}\right)}
        \\
        &\quad+\frac{1}{4\pi} \sum_{p=0}^{m} (-1)^p E_{2kp-1}(a) E_{2k(m-p)-1}(b) \left(\frac{\pi}{\alpha}\right)^{2kp}.
    \end{aligned}
    \end{equation}
Since $\alpha\beta=\pi^2$, by multiplying both sides by $\alpha^{km}$ and rearranging, we have (\ref{Th5}).
\qed

\begin{remark}
    \textup{Theorems \ref{res7} and \ref{res8} correspond to Ramanujan-type identities involving Euler polynomials for the alternating even-order and odd-order Hurwitz kernels under the modular symmetry $\alpha\beta=\pi^2$.}
\end{remark}

\subsection*{Proof of Theorem \ref{res9}}
We give two proofs.

(1) Since $\zeta_{E}(z,x)$  converges absolutely for $\operatorname{Re}(z)>0$, substituting the integral definition of the alternating odd-order Hurwitz kernel into the series on the left-hand side of the theorem and interchanging the infinite sum and the complex integral yields
    \begin{equation}
        \begin{aligned}\label{dingyijifen5}
            &\sum_{n=0}^{\infty}\frac{\left(-1\right)^n
            G_\alpha\left(n+b,a;k\right)}{\left(n+b\right)^{2km}}
            \\&=\sum_{n=0}^{\infty}\frac{\left(-1\right)^n}
            {\left(n+b\right)^{2km}}\left(\frac{1}{2i\pi}
            \int_{\left(c\right)}{\frac{\zeta_E\left(1-s,a\right)}
            {2k \sin{\left(\frac{\pi s}{2k}\right)}}
            \left(\frac{\alpha\left(n+b\right)}{\pi}\right)^{-s}{\rm d}s}\right)
            \\&=\frac{1}{2i\pi}\int_{\left(c\right)}
            \frac{\zeta_E\left(1-s,a\right)\zeta_E\left(2km+s,b\right)}
            {2k \sin{\left(\frac{\pi s}{2k}\right)}}
            \left(\frac{\alpha}{\pi}\right)^{-s}{\rm d}s.
        \end{aligned}
    \end{equation}
    We now construct a rectangular contour and evaluate the integral by shifting the line of integration. In the complex plane, consider the rectangular region bounded by the four line segments 
    $[c-iT,c+iT]$, $[c+iT,d+iT]$, $[d+iT,d-iT]$, $[d-iT,c-iT]$, where $d=-c-2km+1$.
    Since the alternating Hurwitz zeta function is analytic in the entire complex plane, all poles of this integrand come from the denominator
    $2k \sin{\left(\frac{\pi s}{2k}\right)}$. The poles are at integers $-2kp~(p\in\left\{0,1,\ldots,m\right\})$, all of which are simple poles.
    For a simple pole $s=-2kp$, the residue is computed as
\begin{equation}
    \begin{aligned}
        R_{-2kp}&=\lim_{s\rightarrow -2kp}
        {\frac{{\left(s+2kp\right)\zeta}_E\left(1-s,a\right)
        \zeta_E\left(2km+s,b\right)}{2k \sin{\left(\frac{\pi s}{2k}\right)}}
        \left(\frac{\alpha}{\pi}\right)^{-s}}
        \\&=\frac{\zeta_E\left(1+2kp,a\right)\zeta_E\left(2km-2kp,b\right)}
        {\pi \cos{\left(p\pi\right)}}\left(\frac{\alpha}{\pi}\right)^{2kp}
        \\&=\frac{(-1)^p}{\pi} \zeta_E(1+2kp,a) \zeta_E(2km-2kp,b) \left(\frac{\alpha}{\pi}\right)^{2kp}.
    \end{aligned}
\end{equation}
    By Cauchy's residue theorem, we have
    \begin{equation}
    \begin{aligned}\label{cauchy5}
        &\frac{1}{2i\pi}\left[\int_{c-iT}^{c+iT}+\int_{c+iT}^{d+iT}+\int_{d+iT}^{d-iT}+\int_{d-iT}^{c-iT}\right]
\frac{\zeta_E(1-s,a)\zeta_E(2km+s,b)}
{2k \sin\left(\frac{\pi s}{2k}\right)}
\left(\frac{\alpha}{\pi}\right)^{-s}{\rm d}s
\\&=\frac{1}{\pi} \sum_{p=0}^{m} (-1)^p \zeta_E(1+2kp,a) \zeta_E(2km-2kp,b) \left(\frac{\alpha}{\pi}\right)^{2kp}.
    \end{aligned}
    \end{equation}
    Using Lemma \ref{constants} and  \eqref{sin_xiao}, 
    as $T\rightarrow\infty$, the integrals along the horizontal segments tend to $0$. Hence,
    \begin{equation}
        \begin{aligned}\label{0_5}
        &\frac{1}{2i\pi}\int_{\left(c\right)}{\frac{\zeta_E\left(1-s,a\right)
        \zeta_E\left(2km+s,b\right)}{2k \sin{\left(\frac{\pi s}{2k}\right)}}
        \left(\frac{\alpha}{\pi}\right)^{-s}{\rm d}s}\\&=\frac{1}{2i\pi}
        \int_{\left(d\right)}{\frac{\zeta_E\left(1-s,b\right)
        \zeta_E\left(2km+s,a\right)}{2k \sin{\left(\frac{\pi s}{2k}\right)}}
        \left(\frac{\alpha}{\pi}\right)^{-s}{\rm d}s}
        \\
        &\quad+\frac{1}{\pi} \sum_{p=0}^{m} (-1)^p \zeta_E(1+2kp,a) \zeta_E(2km-2kp,b) \left(\frac{\alpha}{\pi}\right)^{2kp}.
    \end{aligned}
    \end{equation}
    After the substitution $s\rightarrow-s-2km+1$, we obtain
    \begin{equation}
        \begin{aligned}\label{tihuan5}
        &\frac{1}{2i\pi}\int_{\left(d\right)}
        {\frac{\zeta_E\left(1-s,a\right)\zeta_E\left(2km+s,b\right)}
        {2k \sin{\left(\frac{\pi s}{2k}\right)}}\left(\frac{\alpha}{\pi}\right)^{-s}{\rm d}s}
        \\&=\left(\frac{\alpha}{\pi}\right)^{2km-1}\left(-1\right)^m
        \frac{1}{2i\pi}\int_{\left(c\right)}{\frac{\zeta_E\left(1-s,b\right)
        \zeta_E\left(2km+s,a\right)}{2k \cos{\left(\frac{\pi\left(s+k-1\right)}{2k}\right)}}
        \left(\frac{\beta}{\pi}\right)^{-s}{\rm d}s}.
    \end{aligned}
    \end{equation}
    Combining  \eqref{0_5} and \eqref{tihuan5}, we have
    \begin{equation}
        \begin{aligned}
        &\frac{1}{2i\pi}\int_{\left(c\right)}
        {\frac{\zeta_E\left(1-s,a\right)\zeta_E\left(2km+s,b\right)}
        {2k \sin{\left(\frac{\pi s}{2k}\right)}}\left(\frac{\alpha}{\pi}\right)^{-s}{\rm d}s}
        \\&=\left(\frac{\alpha}{\pi}\right)^{2km-1}\left(-1\right)^m
        \frac{1}{2i\pi}\int_{\left(c\right)}{\frac{\zeta_E\left(1-s,b\right)
        \zeta_E\left(2km+s,a\right)}{2k \cos{\left(\frac{\pi\left(s+k-1\right)}{2k}\right)}}
        \left(\frac{\beta}{\pi}\right)^{-s}{\rm d}s}
        \\
        &\quad+\frac{1}{\pi} \sum_{p=0}^{m} (-1)^p \zeta_E(1+2kp,a) \zeta_E(2km-2kp,b) \left(\frac{\alpha}{\pi}\right)^{2kp}.
    \end{aligned}
    \end{equation}
    Recalling  \eqref{dingyijifen5} and Definition \ref{alternating_hurwitz_kernel_odd}, we obtain
    \begin{equation}
        \begin{aligned}\label{res9_1}
&\sum_{n=0}^{\infty}\frac{(-1)^n G_\alpha(n+b,a;k)}{(n+b)^{2km}} \\
&= \left(\frac{\alpha}{\pi}\right)^{2km-1} (-1)^m \sum_{n=0}^{\infty}\frac{(-1)^n F_\beta(n+a,b;k)}{(n+a)^{2km}} \\
&\quad+ \frac{1}{\pi} \sum_{p=0}^{m} (-1)^p \zeta_E(1+2kp,a)\zeta_E(2km-2kp,b) \left(\frac{\alpha}{\pi}\right)^{2kp},
    \end{aligned}
    \end{equation}
Since $\alpha\beta=\pi^2$, by multiplying both sides by $\beta^{km}$ and rearranging , we have (\ref{ra-ty-G2}), the functional form of Theorem \ref{res9}.

Then in  \eqref{res9_1}, substituting the series expressions for $G_\alpha(n+b,a;k)$ and $F_\beta\left(n+a,b;k\right)$ (see Definitions \ref{alternating_hurwitz_kernel_even} and \ref{alternating_hurwitz_kernel_odd}) yields
\begin{equation}
    \begin{aligned}
        &\sum_{n=0}^{\infty}\frac{\left(-1\right)^nG_\alpha\left(n+b,a;k\right)}
        {\left(n+b\right)^{2km}}\\&=\sum_{n=0}^{\infty}\frac{\left(-1\right)^n}
        {\left(n+b\right)^{2km}}\left(\frac{-\widetilde{\psi}\left(a\right)}{\pi}
        -\frac{\left(\frac{\alpha\left(n+b\right)}{\pi}\right)^{2k}}{\pi}
        \sum_{i=0}^{\infty}\frac{\left(-1\right)^i}
        {\left(i+a\right)\left(\left(i+a\right)^{2k}
        +\left(\frac{\alpha\left(n+b\right)}{\pi}\right)^{2k}\right)}\right)
        \\&=-
        \frac{1}{\pi}\sum_{n=0}^{\infty}\frac{\left(-1\right)^n\alpha^k}
        {\left(n+b\right)^{2k(m-1)}}\sum_{i=0}^{\infty}\frac{\left(-1\right)^i}
        {\left(i+a\right)\left(\beta^k\left(i+a\right)^{2k}+\alpha^k\left(n+b\right)^{2k}\right)}
        \\
        &\quad-\frac{\widetilde{\psi}\left(a\right)}{\pi}\sum_{n=0}^{\infty}\frac{\left(-1\right)^n
        }{\left(n+b\right)^{2km}},
    \end{aligned}
\end{equation}
\begin{equation}
    \begin{aligned}
        &\sum_{n=0}^{\infty}\frac{\left(-1\right)^n
        F_\beta\left(n+a,b;k\right)}{\left(n+a\right)^{2km}}\\
        &=\sum_{n=0}^{\infty}{\frac{\left(-1\right)^i}
        {\left(n+a\right)^{2km}}\left(-\frac{1}{2\pi\left(\frac{\beta\left(n+a\right)}{\pi}\right)}
        +\frac{1}{\pi}\sum_{i=0}^{\infty}\frac{{(-1)}^n
        \left(\frac{\beta\left(n+a\right)}{\pi}\right)^{2k-1}}{{(i+b)}^{2k}
        +\left(\frac{\beta\left(n+a\right)}{\pi}\right)^{2k}}\right)}
        \\
        &=\sum_{n=0}^{\infty}{\frac{\left(-1\right)^n}{\left(n+a\right)^{2k(m-1)+1}}
        \left(\sum_{i=0}^{\infty}\frac{{(-1)}^i\beta^{k-1}}
        {\alpha^k{(i+b)}^{2k}+\beta^k\left(n+a\right)^{2k}}\right)}\\
        &\quad-\frac{1}{2\beta}\zeta_E\left(1+2km,a\right).
    \end{aligned}
\end{equation}
For the last term of  \eqref{res9_1}
    \begin{align*}
        \sum_{p=0}^{m} (-1)^p \zeta_E(1+2kp,a) \zeta_E(2km-2kp,b) \left(\frac{\alpha}{\pi}\right)^{2kp},
    \end{align*}
when $p=0$, we have
\begin{align*}
    \zeta_E\left(1,a\right)\zeta_E\left(2km,b\right),
\end{align*}
when $p=m$, we have
\begin{align*}
\left(-1\right)^m\zeta_E\left(1+2km,a\right)\zeta_E\left(0,b\right)\left(\frac{\alpha}{\pi}\right)^{km}.
\end{align*}
Thus, multiplying both sides by $\beta^{km}$,  \eqref{res9_1} can be written as
\begin{equation}\label{res9_2}
    \begin{aligned}
        &(-1)^{m-1}\alpha^{km}\sum_{n=0}^{\infty}
        \frac{\left(-1\right)^n\beta^k}{\left(n+a\right)^{2k(m-1)+1}}
        \sum_{i=0}^{\infty}\frac{\left(-1\right)^i}{
        \alpha^k\left(i+b\right)^{2k}+\beta^k\left(n+a\right)^{2k}}
        \\&=-\beta^{km}\sum_{n=0}^{\infty}
        \frac{\left(-1\right)^n\alpha^k}{\left(n+b\right)^{2k(m-1)}}
        \sum_{i=0}^{\infty}\frac{\left(-1\right)^i}{\left(i+a\right)
        \left(\beta^k\left(i+a\right)^{2k}+\alpha^k\left(n+b\right)^{2k}\right)}
        \\
        &\quad+\sum_{p=1}^{m-1}{\left(-1\right)^p\zeta_E\left(2kp+1,a\right)
        \zeta_E\left(2k\left(m-p\right)+1,b\right)\alpha^{kp}\beta^{k\left(m-p\right)}}.
    \end{aligned}
\end{equation}
    After rearranging, we obtain (\ref{Th6}), the series form of Theorem \ref{res9}.

(2) 
Let
$$a_n=\frac{\left(-1\right)^{n-1}}{n-1+a},b_n=\left(-1\right)^{n-1},~
x_n=-\frac{{(n-1+a)}^{2k}}{\alpha^k},~
y_n=\frac{{(n-1+b)}^{2k}}{\beta^k}.$$
Substituting into  \eqref{dirichletzeta}, the corresponding Dirichlet series can be written in terms of $\zeta_E\left(2k+1,x\right)$ and $\zeta_E\left(2k,x\right)$ as follows:
\begin{align*}
    \zeta_{x,a}\left(p\right)&=\sum_{n=1}^{\infty}
    {\frac{\left(-1\right)^{n-1}}{(n-1+a)}
    \frac{1}{\left(-\frac{{(n-1+a)}^{2k}}{\alpha^k}\right)^p}}
    \\&=\left(-\alpha^k\right)^p\sum_{n=1}^{\infty}
    \frac{\left(-1\right)^{n-1}}{{(n-1+a)}^{2kp+1}}
    \\&=\left(-1\right)^p\alpha^{kp}\sum_{n=0}^{\infty}
    \frac{\left(-1\right)^n}{{(n+a)}^{2kp+1}}
    \\&=\left(-1\right)^p\alpha^{kp}\zeta_E\left(2kp+1,a\right),
\end{align*}
\begin{align*}
    \zeta_{y,b}\left(m-p\right)&=\sum_{n=1}^{\infty}
    \frac{\left(-1\right)^{n-1}}{\left(\frac{\left(n-1+b\right)^{2k}}
    {\beta^k}\right)^{m-p}}\\
&=\beta^{k\left(m-p\right)}\sum_{n=1}^{\infty}
    \frac{\left(-1\right)^{n-1}}{\left(n-1+b\right)^{2k\left(m-p\right)}}
    \\
&=\beta^{k\left(m-p\right)}\sum_{n=0}^{\infty}\frac{\left(-1\right)^n}
    {\left(n+b\right)^{2k\left(m-p\right)}}
    \\
&=\beta^{k\left(m-p\right)}\zeta_E\left(2k\left(m-p\right),b\right).
\end{align*}
Furthermore, substituting ${a_n}$, ${b_n}$, ${x_n}$, ${y_n}$ into  \eqref{zetasc} yields the corresponding zeta generating functions:
\begin{align*}
\psi_{x,a}\left(z\right)&=\sum_{i=1}^{\infty}
{\frac{\left(-1\right)^{i-1}}{(i-1+a)}\frac{z}{\left(-\frac{{(i-1+a)}^{2k}}{\alpha^k}-z\right)}}
\\&=-\sum_{i=0}^{\infty}\frac{\left(-1\right)^i\alpha^kz}{(i+a)({(i+a)}^{2k}+\alpha^kz)},
\end{align*}
\begin{align*}
\psi_{y,b}\left(z\right)&=\sum_{i=1}^{\infty}
{\left(-1\right)^{i-1}\frac{z}{\frac{{(i-1+b)}^{2k}}{\beta^k}-z}}
\\
&=\sum_{i=1}^{\infty}\frac{\left(-1\right)^{i-1}\beta^kz}{({(i-1+b)}^{2k}-\beta^kz)} \\
&=\sum_{i=0}^\infty\frac{(-1)^i\beta^k z}{(i+b)^{2k}-\beta^kz}.
\end{align*}
Substituting ${x_n}$, ${y_n}$ gives
\begin{align*}
    \psi_{x,a}\left(y_n\right)=-\sum_{i=0}^{\infty}
    \frac{\left(-1\right)^i\alpha^k{(n-1+b)}^{2k}}{(i+a)({\beta^k\left(i+a\right)}^{2k}+\alpha^k{(n-1+b)}^{2k})},
\end{align*}
\begin{align*}
    \psi_{y,b}\left(x_n\right)=-\sum_{i=0}^{\infty}
    \frac{\left(-1\right)^i\beta^k{(n-1+a)}^{2k}}{{\alpha^k\left(i+b\right)}^{2k}+\beta^k{(n-1+a)}^{2k}}.
\end{align*}
Therefore, applying Lemma \ref{Dirichlet_jishu}, we have
\begin{align*}
&\sum_{p=1}^{m-1}{\left(-1\right)^p\zeta_E\left(1+2kp,a\right)\zeta_E\left(2km-2kp,b\right)\alpha^{kp}\beta^{k(m-p)}}
\\&=\sum_{n=1}^{\infty}\frac{\frac{\left(-1\right)^{n-1}}{n-1+a}}{\left(-\frac{{(n-1+a)}^{2k}}{\alpha^k}\right)^m}
\left(\sum_{i=0}^{\infty}\frac{\left(-1\right)^i\beta^k\left(n-1+a\right)^{2k}}{\alpha^k\left(i+b\right)^{2k}+\beta^k\left(n-1+a\right)^{2k}}\right)
\\
&\quad+\sum_{n=1}^{\infty}\frac{\left(-1\right)^{n-1}}{\left(\frac{{(n-1+b)}^{2k}}{\beta^k}\right)^m}\left(-\sum_{i=0}^{\infty}\frac{\left(-1\right)^i\alpha^k{(n-1+b)}^{2k}}{(i+a)({\beta^k(i+a)}^{2k}+\alpha^k{(n-1+b)}^{2k})}\right)
\\&=(-1)^m\alpha^{km} \sum_{n=0}^{\infty}\frac{\left(-1\right)^n\beta^k}{\left(n+a\right)^{2k(m-1)+1}}\sum_{i=0}^{\infty}\frac{\left(-1\right)^i}{\alpha^k\left(i+b\right)^{2k}+\beta^k\left(n+a\right)^{2k}}
\\
&\quad+\beta^{km}\sum_{n=1}^{\infty}\frac{\left(-1\right)^n\alpha^k}{\left(n+b\right)^{2k\left(m-1\right)}}\sum_{i=0}^{\infty}\frac{\left(-1\right)^i}{(i+a)({\beta^k(i+a)}^{2k}+\alpha^k{(n+b)}^{2k})}.
\end{align*}
This equation is equivalent to  \eqref{res9_2}.
\qed

\begin{remark}
    \textup{Theorem \ref{res9} provides a convolution identity linking the alternating even-order and odd-order Hurwitz kernels under the modular symmetry $\alpha\beta=\pi^2$, with an equivalent series expression.}
\end{remark}

\end{document}